\documentclass[reqno]{amsart}
\usepackage{a4wide}
\usepackage{color}
\usepackage[english]{babel}
\usepackage{amsmath}
\usepackage{amssymb}
\usepackage{enumerate}
\usepackage{verbatim}
\usepackage{esint}
\usepackage[stable]{footmisc}

\newcommand{\setR}{\mathbb{R}}

\newcommand{\setN}{\mathbb{N}}

\newcommand{\flow}{\mathfrak{S}}
\newcommand{\dissip}{\mathfrak{D}}

\newcommand{\ent}{\mathbf{H}}

\newcommand{\prb}{\mathcal{P}_2}

\newcommand{\R}{\mathbb{R}}

\newcommand{\sign}{\mathrm{sign}}
\newcommand{\intT}{\int_0^T}
\newcommand{\intX}{\int_{\R}}
\newcommand{\dd}{\,\mathrm{d}}
\newcommand{\loc}{\text{loc}}

\newcommand{\dff}{\mathrm{D}}

\newcommand{\eps}{\epsilon}

\newcommand{\indy}{\mathbf{1}}
\newcommand{\mom}{\mathbf{m}}

\newcommand{\mina}{\underline{a}}
\newcommand{\maxaxx}{\overline{a_{xx}}}
\newcommand{\rescale}{\mathbf{S}}

\newcommand{\trf}{\mathfrak{T}}

\newcommand{\dens}{{\prb(\setR)}}
\newcommand{\wass}{{\mathbf{W}_2}}

\newcommand{\poten}{\mathcal{F}}
\newcommand{\Kbound}{\mathfrak{K}}

\DeclareMathOperator{\abs}{Abs}
\DeclareMathOperator{\sgn}{Sgn}
\DeclareMathOperator{\stp}{Stp}
\DeclareMathOperator{\heav}{PosPar}
\DeclareMathOperator*{\esssup}{ess\,sup}

\newtheorem{thm}{Theorem}[section]
\newtheorem{prop}[thm]{Proposition}
\newtheorem{lem}[thm]{Lemma}
\newtheorem{cor}[thm]{Corollary}
\newtheorem{rmk}[thm]{Remark}
\newtheorem{defin}[thm]{Definition}


\begin{document}

\title[Entropy solutions via Wasserstein gradient flow]{Curves of steepest descent are entropy solutions for a class of degenerate convection-diffusion equations}

\author{Marco Di Francesco}
\address{Marco Di Francesco, Departament de Matem\`{a}tiques, Universitat Aut\`{o}noma de Barcelona, E-08193 Bellaterra, Catalunya (Spain)}
\email{difrancesco@mat.uab.cat}

\author{Daniel Matthes}
\address{Daniel Matthes, Zentrum Mathematik, Technische Universit\"at M\"unchen, D-85747 Garching bei M\"unchen (Germany)}
\email{matthes@ma.tum.de}
\thanks{Corresponding author: D.~Matthes \texttt{$\langle$matthes@ma.tum.de$\rangle$} Tel: \texttt{+49-89-289-18300} Fax: \texttt{+49-89-289-18308}}

\date{Submission date: 3. August 2012}

\begin{abstract}
  We consider a nonlinear degenerate convection-diffusion equation with inhomogeneous convection
  and prove that its entropy solutions in the sense of Kru\v{z}kov are obtained as the --- a posteriori unique ---
  limit points of the JKO variational approximation scheme for an associated gradient flow in the $L^2$-Wasserstein space.
  The equation lacks the necessary convexity properties
  which would allow to deduce well-posedness of the initial value problem by the abstract theory of metric gradient flows.
  Instead, we prove the entropy inequality directly by variational methods
  and conclude uniqueness by doubling of the variables.
\end{abstract}

\maketitle

\section{Introduction}
The goal of this paper is to show on the example of the degenerate parabolic equation
\begin{equation}\label{eq:y-equation}
  \partial_t u = (u^m)_{yy} + (b(y) u^m)_y,\quad y\in \R,\ \ t\geq0,
\end{equation}
how the solution concepts of \emph{metric gradient flows} and
\emph{entropy solutions} can be combined to obtain global in time
well-posedness of the initial value problem.
Specifically, we define an energy functional $\poten$,
construct time-discrete curves of steepest descent in the landscape of $\poten$ with respect to the $L^2$-Wasserstein metric
by means of the JKO variational scheme \cite{JKO},
and show that these curves converge in the limit of continuous time to entropy solutions \`{a} la Kru\v{z}kov \cite{KruvzkovEntropy,kruzkov}
(or, more precisely, in the sense of Carrillo \cite{Car99}) for \eqref{eq:y-equation}.
This way, existence of solutions to \eqref{eq:y-equation} is obtained by a variational method,
and their uniqueness follows from the theory of entropy solutions.

In \eqref{eq:y-equation}, the exponent $m>1$ determines the degeneracy
of the diffusion for vanishing densities, and
\begin{equation}\label{eq:assumption_b}
  b\in L^1(\R)\cap W^{1,\infty}(\R)
\end{equation}
is a given coefficient modelling heterogeneity in the convection part.
As initial condition for \eqref{eq:y-equation}, we prescribe a datum $u^0\in L^1(\setR)\cap L^\infty(\setR)$ of finite second moment.

The link to metric gradient flows is established as follows.
There is a spatial coordinate change $y=T(x)$ and an according transformation $\rho(t,x)=T'(x)u(t,T(x))$ (see subsection \ref{subsec:coordinate} below) such that
all sufficiently regular solutions $u$ to \eqref{eq:y-equation} become solutions of
\begin{align}
  \label{eq:x-equation}
  \partial_t\rho = \big( \rho [a(x)\rho^{m-1}]_x \big)_x,
\end{align}
and vice versa;
here $a\in W^{2,\infty}(\setR)$ is a strictly positive function, determined from $b$ via $T$.
Well-known formal arguments, see e.g. \cite{villani_book}, indicate that
the time-dependent density functions $\rho(t)$ satisfying \eqref{eq:x-equation}
are ``curves of steepest descent'' in the energy landscape of the entropy functional
\begin{equation}
  \label{eq:poten}
  \poten[\rho]:=\frac{1}{m}\intX a(x) \rho^m \dd x,
\end{equation}
with respect to the $2$-Wasserstein distance.

Indeed, if $\poten$ would be a geodesically $\lambda$-convex functional,
then the theory of $\lambda$-contractive gradient flows --- see e.g. \cite{AGS} --- could be applied
to conclude the existence of a unique gradient flow for $\poten$ in the space of probability measures.
This flow's curves would be weak solutions to \eqref{eq:x-equation},
and by inversion of the coordinate transformation above, we could conclude well-posedness for \eqref{eq:y-equation}.
Unfortunately, $\poten$ does apparently not have the required convexity property;
see Section \ref{sec:gradient_flow} for details.

Nevertheless, the variational structure behind \eqref{eq:x-equation} is on the basis
for our proof of \emph{existence} of solutions to \eqref{eq:y-equation}.
Specifically, we employ the JKO (or ``minimizing movement'') approximation scheme
to obtain time-discrete curves $\rho_\tau$ of steepest descent for $\poten$.
We prove that any family of such discrete approximations $\rho_\tau$ possesses a weak time-continuous limit curve $\rho_*$.
By inversion of the coordinate change, this provides a candidate $u_*$ for a solution to \eqref{eq:y-equation}.
We emphasize that this construction does not require geodesic $\lambda$-convexity for $\poten$;
boundedness from below, coercivity and lower semi-continuity are sufficient.

At this point, the variational framework of minimizing movements provides
strong tools which allow us to prove that $u_*$ is actually an
entropy solution for \eqref{eq:y-equation} in the sense of Carrillo \cite{Car99}:
we show that $u_*^m\in L^2_\loc(0,\infty;H^1(\setR))$, and that for every
non-negative test function $\varphi\in C^\infty_c(\setR_+\times\setR)$
and for every $k\in\setR_+$:
\begin{align}
  \label{eq:preei}
  \begin{split}
    & \intT\intX |u_*-k|\varphi_t\dd y\dd t \\ 
    & - \intT\intX \sgn(u_*-k)
    \Big(\big[(u_*^m)_y + b(u_*^m-k^m)\big]\varphi_y - b_yk^m \varphi \Big)\dd y\dd t
    \ge \mathfrak{D}(u_*)[\varphi] \ge 0.
  \end{split}
\end{align}
We also obtain a non-trivial lower bound on the dissipation $\mathfrak{D}(u_*)$;
see Proposition \ref{prop:karlsen} for details.

The derivation of the entropy inequality \eqref{eq:preei} is the core element of our proof.
First, a time-discrete version of this estimate is proven directly for the JKO scheme, and this is passed to the limit.
In the derivation of the discrete estimates, the key idea is --- as usual --- to choose appropriate variations of the minimizers.
Here we build on the ideas
that have already been employed in \cite{JKO} for the derivation of the weak formulation of the linear Fokker-Planck equation
and have been generalized later in \cite{matthes_mccann_savare} to the ``flow interchange lemma'', see Lemma \ref{lem:fi}:
variations are performed by an auxiliary gradient flow, which --- in contrast to the gradient flow of \eqref{eq:poten} itself --- is $\lambda$-convex
and thus satisfies certain variational inequalities.


Once that \eqref{eq:preei} has been establish,
we adapt the doubling of the variables method in \cite{karlsen2003} to our case
and show that $u_*$ is actually the \emph{unique} entropy solution for the given initial condition $u^0$.
A posteriori, we conclude uniqueness of the limit curve $\rho_*$ of the approximation scheme as well.

The following theorem summarizes our main result in an informal way;
the precise statement is given in Theorem \ref{thm:main} in Section \ref{subsec:results}.
For a possible slight generalization avoiding the scaling, see Remark \ref{remark:alternative_result}.
\begin{thm}
  Let an initial condition $u^0\in L^\infty(\setR)$ of finite second moment be given.
  Then every curve obtained from the JKO approximation for $\poten$ in the limit of continuous time
  corresponds --- by a chance of coordinates --- to the unique entropy solution for \eqref{eq:y-equation}.
\end{thm}
A related observation about the connection between entropy solutions and gradient flows has been made recently
by Gigli and Otto \cite{GigliOtto} in the context of the inviscid Burger's equation, see also previous results in \cite{dafermos_entropy,Otto-relax}.
The interpretation of the coincidence between entropy solutions and gradient flows is that
both types of solutions can be characterized by diminishing an underlying entropy functional ``as fast as possible''.
For previous results on the well-posedness of scalar conservation laws in Wasserstein spaces, we refer to \cite{brenier,CDT}.

Apart from revealing an interesting connection between the two seemingly unrelated theories of entropy solutions and Wasserstein gradient flows,
our simple example indicates a possible general strategy to prove existence and uniqueness of certain nonlinear evolution equations
which can be cast in the form of a Wasserstein gradient flow of a functional
which is \emph{not} necessarily geodesically $\lambda$-convex.
First, use methods from the calculus of variations to construct a candidate for a solution;
a priori, there might be several.
Second, show that this candidate is an entropy solution by deriving further a priori estimates in the variational framework.
Third, conclude uniqueness of the entropy solution.
This strategy provides a new method to obtain entropy solutions,
alternative e.g. to the classical vanishing viscosity approach \cite{Daf00},
to the wave-front-tracking algorithm \cite{dafermos_wft},
or to semigroup theory \cite{Cra72}.
In particular, the variational approach does not require the solution of auxiliary regularized problems.

To conclude this introduction,
we stress again that the goal of this paper is to establish a link between gradient flows and entropy solutions
but not to refine the results on existence and uniqueness of weak solutions to \eqref{eq:y-equation}.
This said, we remark that --- as far as our model equation \eqref{eq:y-equation} is concerned ---
uniqueness of weak $L^1$-solutions could be proven also by other means, e.g. with the methods developed in \cite{otto97}.

The paper is organized as follows.
In Section \ref{sec:statement}, we formulate the problem and state our main results.
Section \ref{sec:gradient_flow} recalls some basic fact from the theory of gradient flows.
In Lemma \ref{lem:z}, we provide a result on the contractivity of quite general reaction-diffusion equations in the Wasserstein metric
which might be of independent interest.
In Section \ref{sec:time_discrete} we prove convergence of the scheme and show that the limit curves are entropy solutions.
Section \ref{sec:uniqueness} contains (for the sake of completeness) the uniqueness proof of entropy solutions.

\medskip

\section*{Notation}

\subsection*{Measures, densities and Wasserstein distance}
Here $\dens$ denotes the set of probability \emph{densities} $\rho\in L^1(\setR)$ with finite second moment $\int x^2\rho(x)\dd x$.
Note that the symbol $\dens$ is frequently used in the literature for the (wider) space of probability \emph{measures} on $\R$,
which we shall denote by $\overline{\dens}$ instead.
Given a measure $\mu\in\overline{\dens}$,
its probability distribution function $U:\setR\to[0,1]$ and pseudo-inverse $G:[0,1]\to\setR\cup\{\pm\infty\}$, respectively,
are given by
\begin{align*}
  U(x) = \mu\big((-\infty,x)\big), \quad
  G(\omega) = \sup\big\{ x\in\setR\,\big|\,U(x)\le\omega\big\}.
\end{align*}
The space $\overline{\prb(\R)}$ is endowed with the $L^2$-Wasserstein distance $\wass$,
defined by
\begin{align}
  \label{eq:wassbyidf}
  \wass(\mu,\tilde\mu) = \bigg( \int_0^1 \big[G(\omega)-\tilde G(\omega)\big]^2\dd\omega \bigg)^{1/2},
\end{align}
where $G$ and $\tilde G$ are the pseudo-inverse distribution functions of $\mu$ and $\tilde\mu$, respectively.
The pair $(\overline{\dens},\wass)$ is a complete metric space.
We refer to \cite{villani_book} for a more detailed explanation.

\subsection*{Mollifications}
We will frequently use the following elementary functions:
the absolute value $\abs(x):=|x|$,
the positive part $\heav(x)=(x)_+=(x+|x|)/2$,
the sign $\sgn(x)=x/|x|$ (with $\sgn(0)=0$),
and the unit step function $\stp(x)=(x+|x|)/2|x|$ (with $\stp(0)=0$).
In fact, we will mostly use their regularized versions obtained by \emph{mollification}:
denote by $\delta_1:\setR\to\setR$ the standard mollifier
\begin{align}
  \delta_1(y) = \begin{cases}
    Z^{-1}\exp[-1/(1-y^2)] & \text{for all $y\in(-1,1)$}, \\ 0 & \text{otherwise},
  \end{cases}\label{eq:mollifier}
\end{align}
where $Z>0$ is chosen s.t.\ $\delta_1$ has unit integral. For
$\eps>0$, define the $\eps$-mollifier $\delta_\eps:\setR\to\setR$ by
$\delta_\eps(y)=\eps^{-1}\delta_1(\eps^{-1}y)$.
Accordingly, we denote by $\abs_\eps=\abs\star\delta_\eps$ and $\heav_\eps=\heav\star\delta_\eps$
the $\eps$-mollifications of $\abs$ and $\heav$, respectively.
Notice that $\sgn_\eps'=2\delta_\eps$ and $\stp_\eps'=\delta_\eps$, which means in particular that $\sgn_\eps$ and $\stp_\eps$ are non-decreasing functions.

\section{Statement of the problem and results}\label{sec:statement}
\subsection{Coordinate transformation}\label{subsec:coordinate}
In this subsection we shall establish the correspondence between equations \eqref{eq:y-equation} and \eqref{eq:x-equation}. More precisely, we prove the following proposition.

\begin{prop}\label{prop:change_var}
  Let $b\in W^{1,\infty}(\R)\cap L^1(\R)$.
  Then, there exist a function $a$ with the properties
  \begin{enumerate}
  \item [(a1)] $a(x)\geq\mina>0$ for all $x\in \R$,
  \item [(a2)] $a\in W^{2,\infty}(\R)$,
  \end{enumerate}
  a bijective change of coordinates $y=\trf(x)$ on $\R$ and a corresponding transformation $\rescale:L^1(\setR)\to L^1(\setR)$ with
  \begin{equation}
    \label{eq:defscaling}
    \rescale[u](x) := \trf'(x) u(\trf(x))
  \end{equation}
  such that the transformation $\rho(t):=\rescale[u(t)]$ of an arbitrary weak solution $u$ to \eqref{eq:y-equation} with initial datum $u_0\in L^1(\R)$,
  satisfies equation \eqref{eq:x-equation} with initial datum $\rho^0=\rescale[u^0]$.
\end{prop}
\begin{proof}
  Choose a positive constant $\alpha_0$,
  and define
  \begin{equation}\label{eq:alpha}
    \alpha(y)=\alpha_0\exp\left( -\frac{(m-1)}{2m} \int_0^y b(\eta)d\eta\right).
  \end{equation}
  The assumptions on $b$ ensure $\alpha \in W^{2,\infty}(\R)$ and $\alpha\geq \underline{\alpha}$ for some $\underline{\alpha}>0$.
  Therefore, the initial value problem
  \begin{equation}
    \label{eq:ivpa}
    \trf_x(x)=\alpha(\trf(x)) \quad \text{for $x\in\setR$}, \qquad \trf(0)=0
  \end{equation}
  admits a unique global solution $\trf:\setR\to\setR$, with $\trf_x(x)\geq\underline{\alpha}$.
  By the chain rule, $\trf$ satisfies further
  \begin{equation}
    \label{eq:change_cauchy}
    \trf_{xx}(x) = \alpha'(\trf(x))\trf_x(x),
  \end{equation}
  and in view of \eqref{eq:alpha}, we also have
  \begin{align}
    \label{eq:alphaderivative}
    \alpha'\circ\trf = -\frac{m-1}{2m}(b\circ\trf)\trf_x .
  \end{align}
  Now, let $u$ be a weak solution to \eqref{eq:y-equation},
  that is
  \begin{align}
    \label{eq:uweak}
    -\int_0^\infty\intX u\varphi_t\dd y\dd t = \int_0^\infty\intX u^m\varphi_{yy}\dd y\dd t - \int_0^\infty\intX u^mb\varphi_y\dd y\dd t
  \end{align}
  for all test functions $\varphi\in C^\infty_c(\setR_+\times\setR)$.
  Let $\rho(t)=\rescale[u(t)]$ for all $t\ge0$, which means that
  \begin{align}
    \label{eq:subst}
    u(\trf(x),t) = \frac{\rho(x,t)}{\trf_x(x)}
  \end{align}
  for all $x\in\setR$.
  Further, for given $\varphi$, define the transformed test function $\psi\in C^\infty_c(\setR_+\times\setR)$ by
  \begin{align*}
    \psi(x,t) = \varphi(\trf(x),t).
  \end{align*}
  Using \eqref{eq:ivpa} and \eqref{eq:change_cauchy}, one easily verifies that
  \begin{align}
    \label{eq:substtest}
    \varphi_y(\trf(x),t)\trf_x(x) = \psi_x(t,x), \quad \varphi_{yy}(\trf(x),t)\trf_x(x)^2 = \psi_{xx}(x,t)-\psi_x(x,t)\alpha'(\trf(x)).
  \end{align}
  Substitute \eqref{eq:subst} into \eqref{eq:uweak},
  perform the change of variables $(y,t)=(\trf(x),t)$ under each of the integrals,
  and simplify the expressions containing test functions using \eqref{eq:substtest}.
  This yields
  \begin{align}
    \label{eq:weakrho}
    \begin{split}
      -\int_0^\infty\intX \rho\psi_t\dd x\dd t
      &= \int_0^\infty\intX \rho^m\psi_{xx}\trf_x^{-(m+1)}\dd x\dd t \\
      & \qquad - \int_0^\infty\intX\rho^m\psi_x\big[(\alpha'\circ\trf)\trf_x^{-(m+1)}+(b\circ\trf)\trf_x^{-m}\big]\dd x\dd t.
    \end{split}
  \end{align}
  Defining $a:\setR\to\setR$ by
  \begin{align}
    \label{eq:defa}
    a(x) = \frac{m}{m-1}(\alpha\circ\trf(x))^{-(m+1)},
  \end{align}
  we find by direct calculations that
  \begin{align}
    \label{eq:rela}
    \frac1m a_x
    \stackrel{\eqref{eq:ivpa}}{=} -\frac{m+1}{m-1}(\alpha'\circ\trf)\trf_x^{-(m-1)}
    \stackrel{\eqref{eq:alphaderivative}}{=} (\alpha'\circ\trf)\trf_x^{-(m+1)}+(b\circ\trf)\trf_x^{-m}.
  \end{align}
  Substitute \eqref{eq:defa} and \eqref{eq:rela}, respectively, in the first and the second integral on the right hand side of \eqref{eq:weakrho},
  to find
  \begin{align*}
    -\int_0^\infty\intX \rho\psi_t\dd x\dd t
    = \frac{m-1}{m}\int_0^\infty\intX a\rho^m\psi_{xx}\dd x\dd t
    - \frac1m \int_0^\infty\intX a_x\rho^m\psi_x\dd x\dd t.
  \end{align*}
  It is easily checked that this is a weak formulation of \eqref{eq:x-equation}.
\end{proof}
\begin{rmk}
  The function $a$ in Proposition \ref{prop:change_var} is not uniquely determined:
  different choices of $\alpha_0>0$ in \eqref{eq:alpha} change $a$ by a positive factor.
  On the other hand, if the function $a$ is given,
  then $b$ and the corresponding change of variable $y=\trf(x)$ can be recovered from $a$ in a unique way.
  Indeed, it follows from \eqref{eq:ivpa} and \eqref{eq:defa} that
  \begin{align}
    \label{eq:defT}
    \trf(x)= \int_0^x \left[\frac{m-1}{m}a(\xi)\right]^{-\frac{1}{m+1}} \dd\xi,
  \end{align}
  while \eqref{eq:defa} and \eqref{eq:alphaderivative} imply
  \begin{align}
    \label{eq:bfroma}
    b\circ\trf = \frac{2m}{m^2-1}(\log a)_x.
  \end{align}
\end{rmk}
\begin{rmk}
  The scaling $\trf$ is still well-defined if the assumption $b\in L^1(\R)$ is relaxed to $b\propto|y|^{-1}$ as $|y|\rightarrow +\infty$.
  However, assumptions (a1) \& (a2) may be not satisfied in those cases.
  Therefore we require $b\in L^1(\R)$.
  Let us also emphasise that the case $b\equiv\mathrm{const}$ cannot be included,
  since the solution to \eqref{eq:change_cauchy} would blow up at some negative $x$.
\end{rmk}

\subsection{Entropy solutions for \eqref{eq:y-equation}}\label{subsec:entropy}

The notion of entropy solution for \eqref{eq:y-equation} used here is a variant of the one
originally introduced by Carrillo \cite{Car99} and later adapted by Karlsen et al. \cite{karlsen2003}.
To motivate this definition, consider the usual viscous approximation of \eqref{eq:y-equation},
\begin{align}
  \label{eq:viscous1}
  \partial_t u = (b(y) u^m)_y + (u^m)_{yy} + \nu u_{yy},
\end{align}
which possesses smooth and classical solutions $u_\nu$ for every $\nu>0$,
provided that the initial condition $u^0$ is smooth enough.
The following formal considerations are made under the hypothesis that $u_\nu$ converges
--- locally uniformly on $\setR_+\times\setR$ and in $L^2_{\loc}([0,T[\times \R)$ ---
to a limit function as $\nu\downarrow0$.

Similarly as in the classical approach by Kru\v{z}kov \cite{kruzkov} for scalar conservation laws,
we would like to derive from \eqref{eq:viscous1} an evolution \emph{in}equality for all functions of the form $|u-k|$ with given $k\in\setR_+$.
As usual, the calculations are carried out with a suitable approximation of $|u-k|$;
in our case, we multiply equation \eqref{eq:viscous1} by $\sgn_\eps(u^m-k^m)$ with a mollification parameter $\eps>0$
and rewrite it in the following way:
\begin{align}
  \label{eq:prepareentropy}
  \sgn_\eps(u^m-k^m)\partial_tu
  = \sgn_\eps(u^m-k^m)\big[(u^m)_{y} +b(u^m-k^m) + \nu u_y\big]_y
  + \sgn_\eps(u^m-k^m)b_yk^m .
\end{align}
Now let $T>0$, integrate \eqref{eq:prepareentropy} against a non-negative test function $\varphi \in C^\infty_c(]0,T[\times \R)$,
and integrate by parts in the first term on the right-hand side:
\begin{align}
  \label{eq:ent1}
  - \intT\intX  &\varphi\sgn_\eps(u^m-k^m)\partial_tu \dd y\dd t =  \\
  \label{eq:ent2}
  & \intT\intX \sgn_\eps(u^m-k^m)
  \Big(\big[(u^m)_y + b(u^m-k^m) + \nu u_y\big]\varphi_y - b_yk^m \varphi \Big)\dd y\dd t \\
  \label{eq:ent3}
  & + \intT\intX \varphi b\sgn_\eps'(u^m-k^m)(u^m-k^m)(u^m)_y \dd y\dd t \\
  \label{eq:ent4}
  & + \intT\intX \sgn_\eps'(u^m-k^m)\big[(u^m)_y^2 + m\nu u^{m-1}u_y^2\big]\varphi\dd y\dd t.
\end{align}
For further simplification, observe that inside the integrand in \eqref{eq:ent1},
\begin{align*}
  \lim_{\eps\downarrow0}\sgn_\eps(u^m-k^m)\partial_t u = \sgn(u-k)\partial_tu = \partial_t|u-k|
\end{align*}
at every point $(t,y)$ with $u(t,y)\neq k$.
Likewise, in \eqref{eq:ent2},
the term with $\nu$ becomes $\sgn(u-k)u_y = |u-k|_y$ in the limit $\eps\downarrow0$.
After integration by parts, its contribution is
\begin{align*}
  -\nu\intT\intX |u-k|\varphi_{yy}\dd y\dd t.
\end{align*}
This term is negligible in the limit $\nu\downarrow0$.
Next, the integrand in \eqref{eq:ent3} can be rewritten as
\begin{align*}
  \varphi b [R_\eps(u^m-k^m)]_y \quad \text{with} \quad
  R_\eps(s) = \int_0^s r\sgn_\eps'(r)\dd r.
\end{align*}
Since $R_\eps$ converges to zero uniformly for $\eps\downarrow0$,
it follows --- after an integration by parts --- that the integral \eqref{eq:ent3} vanishes in that limit.
Finally, the integral \eqref{eq:ent4} obviously gives a non-negative contribution.
Thus in the limit $\eps\downarrow0$, and after an integration by parts with respect to time in \eqref{eq:ent1},
the equality \eqref{eq:ent1}--\eqref{eq:ent4} implies that
\begin{equation}
  \label{eq:entropyformulation}
  \begin{split}
    \intT\intX |u-k|\varphi_t\dd y\dd t 
    & \ge \intT\intX \sgn(u-k)
    \Big(\big[(u^m)_y + b(u^m-k^m)\big]\varphi_y - b_yk^m \varphi \Big)\dd y\dd t \\
    & \qquad + \limsup_{\eps\downarrow0} \intT\intX \sgn_\eps'(u^m-k^m)\big[(u^m)_y\big]^2\varphi\dd y\dd t.
  \end{split}
\end{equation}
We take this as the defining inequality for entropy solutions.
%
%
%
\begin{defin}[Definition of entropy solution]\label{def:entropy_sol_1}
  Let $u^0 \in L^1\cap L^\infty(\R)$.
  A non-negative measurable function $u:\setR_+\times\setR\to\setR$ is an \emph{entropy solution} to \eqref{eq:y-equation} with initial condition $u^0$
  if $u\in L^1\cap L^\infty(]0,T[\times\setR)$ for all $T>0$,
  if $u^m\in L^2_\loc(0,\infty;H^1(\setR))$,
  if $u(t)\to u^0$ in $L^1(\setR)$ as $t\downarrow0$,
  and if inequality \eqref{eq:entropyformulation} is satisfied
  for all nonnegative test functions $\varphi\in C^\infty_c(\setR_+\times\setR)$, and for all $k\in\setR_+$.
\end{defin}
\begin{rmk}
  The inclusion of the very particular dissipation term in inequality \eqref{eq:entropyformulation} seems a bit ad hoc.
  This term, however, plays a key role in the proof of uniqueness, see Section \ref{sec:uniqueness}.
  Moreover, the choice of this term is less arbitrary than it appears:
  \emph{any} smooth and uniformly convergent approximation of the sign function
  could be used in place of the mollification $\sgn_\eps$ there.

  Finally, we remark that a substantial part of the article \cite{karlsen2003} is devoted to proving that
  all functions $u$ of a certain regularity which satisfy \eqref{eq:entropyformulation} \emph{without} the dissipation term
  actually satisfy it also \emph{with} the dissipation term.
  For us, the dissipation term results very naturally from our construction of solutions $u$.
\end{rmk}
\begin{rmk}
  \label{rmk:tensor}
  For a function $u$ of the specified regularity, in order to be an entropy solution it is sufficient
  that inequality \eqref{eq:entropyformulation} is satisfied for all test function $\varphi$ of the form $\varphi(t,y)=\theta(t)\phi(y)$
  with arbitrary non-negative functions $\theta\in C^\infty_c(\setR_+)$ and $\phi\in C^\infty_c(\setR)$.
  This follows immediately since the latter products lie dense in $C^\infty_c(\setR_+\times\setR)$, see e.g. \cite[Theorem 4.3.1]{friedlander}.
\end{rmk}
In difference to \cite{karlsen2003},
we do not require --- and, in fact, cannot prove --- that $u$ is a continuous curve in $L^1(\setR)$ for $t>0$.

\subsection{The JKO scheme}\label{subsec:JKO}
The JKO scheme \cite{JKO} is a variant of the time-discrete implicit Euler approximation
for the solution of gradient flows in the non-smooth metric setting of the $L^2$-Wasserstein distance.
In its core, it is a special case of De Giorgi's minimizing movement scheme \cite{degiorgi}; see the book \cite{AGS} for an extensive theory.

We apply the JKO scheme
to the functional $\poten$ defined in \eqref{eq:poten} with $a$ defined implicitly in Proposition \ref{prop:change_var}
and inital condition $\rho^0:=\rescale[u^0]$, see \eqref{eq:defscaling}.
To this end, let a time step $\tau>0$ be given.
For every $\sigma \in \dens$, introduce the associated Yoshida penalization $\poten_\tau(\cdot;\sigma)$ of $\poten$ by
\begin{align*}
  \poten_\tau(\rho;\sigma) =  \frac1{2\tau}\wass(\rho,\sigma)^2 + \poten(\rho).
\end{align*}
Let further an initial condition $\rho^0\in \dens$ with $\poten(\rho^0)<+\infty$ be given.
We define a sequence of densities $\rho_\tau^n\in\dens$ inductively as follows:
\begin{enumerate}
\item $\rho_\tau^0:=\rho^0$.
\item For $n\ge1$, let $\rho_\tau^n\in\dens$ be the (unique global) minimizer of $\poten_\tau(\cdot;\rho_\tau^{n-1})$.
\end{enumerate}
In Lemma \ref{lem:scheme_well_posed} we prove well-definiteness of this scheme.
In the following, we denote by $\bar\rho_\tau:[0,\infty[\to\dens$ the piecewise constant interpolation of the sequence $(\rho_\tau^n)_{n\in\setN}$,
with
\begin{align}
  \label{eq:interpol}
  \bar\rho_\tau(t) = \rho_\tau^n \quad \text{for $(n-1)\tau<t\le n\tau$}.
\end{align}

\subsection{Main results}
\label{subsec:results}
We are now in the position to give the precise statement of the two main results of this paper.
\begin{thm}\label{thm:main}
  Let $\rho^0\in\dens\cap L^\infty(\R)$, and let $a$ satisfy conditions (a1)\&(a2) of Proposition \ref{prop:change_var}.
  Define discrete curves $\bar\rho_\tau:[0,\infty[\to\dens$ by means of the JKO scheme from Section \ref{subsec:JKO} above.
  Then every vanishing sequence $(\tau_k)_{k\in\setN}$ of time steps contains a subsequence (not relabeled)
  such that the $\rho_{\tau_k}$ converge
  --- in $L^m(]0,T[\times\setR)$, and also uniformly in $\wass$ on each time interval $[0,T]$ ---
  to a curve $\rho_*:[0,\infty[\to\dens$ that is continuous with respect to $\wass$.
  The rescaled function $u_*=\rescale^{-1}[\rho_*]$ is an entropy solution to \eqref{eq:y-equation} in the sense of Definition \ref{def:entropy_sol_1},
  with initial condition $u^0=\rescale^{-1}[\rho^0]$.
\end{thm}
The proof of Theorem \ref{thm:main} is performed in Section \ref{sec:time_discrete}.
\begin{thm}\label{thm:uniqueness}
  For each initial condition $u^0\in\dens\cap L^\infty(\R)$,
  there exists only one entropy solution to \eqref{eq:y-equation} in the sense of Definition \ref{def:entropy_sol_1}.
  Consequently, \emph{every} sequence of time-discrete approximations $\bar\rho_{\tau_k}$ converges to the same limit $\rho_*$.
\end{thm}
The proof of Theorem \ref{thm:uniqueness} is taken from \cite{karlsen2003}, with minor modifications.
In order to make the present paper self-contained, we outline the argument and review the relevant calculations in Section \ref{sec:uniqueness}.

\begin{rmk}\label{remark:alternative_result}
The notion of entropy solution can be posed as well for the equation \eqref{eq:x-equation} without passing through the scaling established in Proposition \ref{prop:change_var}, and a uniqueness theorem in the spirit of Theorem \ref{thm:uniqueness} can be derived with arguments in \cite{karlsen_ohlberger}. Nevertheless, we opted for developing our theory based on the notion of entropy solution for the equation \eqref{eq:y-equation}, which can be seen as a porous medium equation with a nonlinear convection perturbation, and it is therefore of interest in the applications.
\end{rmk}

\section{Convex functionals and contractive gradient flows}\label{sec:gradient_flow}
Contractive gradient flows (or ``$\kappa$-flows'', see Definition \ref{def:k_flow} below) constitute our key tool
for the derivation of a priori estimates on solutions $\rho$ to \eqref{eq:x-equation},
or, equivalently, for solutions $u$ to \eqref{eq:y-equation}.
First, we recall the definition of $\kappa$-flows and their relation to geodesically $\lambda$-convex functionals.
Then we prove $\kappa$-contractivity of a class of flows that is relevant for our needs.
We refer to \cite{AGS,villani_book,villani_book_2} as references on the general gradient flow theory on Wasserstein spaces.

\subsection{Fundamental definitions and relations}
%
 A curve $\rho:I\to\dens$ is called \emph{absolutely continuous in $\wass$} on the interval $I\subset\setR$
 if there exists a function $g\in L^1_\loc(I)$ such that
 \begin{equation*}
   \wass(\rho(t),\rho(s))\leq \Big|\int_s^t g(\tau)d\tau\Big| \quad \text{for all $t,s\in I$}.
 \end{equation*}
 An absolutely continuous curve $\rho:[0,1]\to\dens$ is a \emph{constant speed geodesic} if
 \begin{align*}
   \wass(\rho(s),\rho(t)) = |t-s|\wass(\rho(0),\rho(1)) \quad \text{for all $t,s\in[0,1]$}.
 \end{align*}

\begin{defin}[$\kappa$-flow]\label{def:k_flow}
  A semigroup $\flow_\Psi:[0,\infty[\times\dens\to\dens$ is a \emph{$\kappa$-flow} for a functional $\Psi:\dens\to\setR\cup\{+\infty\}$ with respect to $\wass$
  if, for arbitrary $\rho\in\dens$,
  the curve $s\mapsto\flow_\Psi^s[\rho]$ is absolutely continuous on $[0,\infty[$
  and satisfies the evolution variational inequality (EVI)
  \begin{align}
    \label{eq:evi}
    \frac12\frac{\dd^+}{\dd\sigma}\Big|_{\sigma=s} \wass\big( \flow_\Psi^\sigma[\rho],\tilde\rho \big)^2
    + \frac\kappa2\wass(\flow_\Psi^s[\rho],\tilde\rho)^2 \le \Psi(\eta) - \Psi(\flow^\Psi_s[\rho])
  \end{align}
  for all $s>0$,
  with respect to every comparison measure $\tilde\rho\in\dens$ for which $\Psi(\tilde\rho)<\infty$.
\end{defin}
\begin{rmk}
  The symbol $\dd^+/\dd\sigma$ stands for the limit superior of the respective difference quotients,
  and equals to the derivative if the latter exists.
\end{rmk}
The fact that a functional $\Psi$ admits a $\kappa$-flow is equivalent to the $\lambda$-convexity of $\Psi$ along geodesics.
The characterization of $\kappa$-flows by convexity will not play a role in our further considerations,
but we cite the respective result for the sake of completeness, see \cite{AGS} for further details.
\begin{thm}\label{thm:convexity}
  Assume that the functional $\Psi:\dens\rightarrow\R\cup\{+\infty\}$ is \emph{$\lambda$-convex (along geodesics)},
  with a modulus of convexity $\lambda \in \R$.
  That is, along every constant speed geodesic $\rho:[0,1]\to\dens$,
  \begin{equation}\label{eq.definition_convexity}
    \Psi[\rho(t)] \leq (1-t)\Psi[\rho(0)] + t\Psi[\rho(1)] -\frac{\lambda}{2} t(1-t) \wass(\rho(0),\rho(1))^2
  \end{equation}
  holds for every $t\in[0,1]$.
  Then $\Psi$ possesses a uniquely determined $\kappa$-flow, with some $\kappa\ge\lambda$.
  On the other hand, if a functional $\Psi$ possesses a $\kappa$-flow, then it is $\lambda$-convex, with some $\lambda\ge\kappa$.
\end{thm}

\subsection{A special class of $\kappa$-flows}\label{subsec:contractive}
In this section, we derive a sufficient condition under which the functional
\begin{align}
  \label{eq:generalpsi}
  \Psi(\eta) = \int_\setR F(x,\eta(x))\dd x
\end{align}
for a sufficiently smooth function $F:\setR\times\setR_+\to\setR$ admits a $\kappa$-flow.
For a concise formulation of that condition, we introduce the adjoint function $H:\setR\times\setR_+\to\setR$
by
\begin{align}
  \label{eq:h}
  H(x,\xi) = \xi F(x,1/\xi),
\end{align}
which satisfies
\begin{align}
  \label{eq:hprop}
  H_x(x,\xi) = \xi F_x(x,1/\xi),\quad
  H_\xi (x,\xi) = F(x,1/\xi) - 1/\xi\,F_\eta(x,1/\xi).
\end{align}
\begin{lem}
  \label{lem:z}
  Let $F\in C^2(\setR\times\setR_+)$ be given,
  and assume that there exist constants $0<c<C$ such that
  \begin{align}
    \label{eq:elliptic2}
    c\le\eta F_{\eta\eta}(x,\eta)\le C \quad \text{and} \quad |F_{x\eta}(x,\eta)|\leq C
    \quad \text{for all $(x,\eta)\in\setR\times\setR_+$}.
  \end{align}
  Assume further that there is some $\kappa\in\setR$ such that
  \begin{align}
    \label{eq:joint}
    (x,\xi)\mapsto H(x,\xi)-\frac\kappa2 x^2
  \end{align}
  is (jointly) convex on $\setR\times\setR_+$.
  Then, the solution operator $\flow_\Psi$ of the evolution equation
  \begin{align}
    \label{eq:evol}
    \partial_t \eta = \dff_x( \eta \dff_x[F_\eta(x,\eta)])
  \end{align}
  is a $\kappa$-flow for the functional $\Psi$ from \eqref{eq:generalpsi}.
\end{lem}
The convexity condition \eqref{eq:joint} can be rephrased as
\begin{align}
  \label{eq:convex}
  \dff^2 H =
  \begin{pmatrix}
    H_{xx}-\kappa & H_{x\xi} \\ H_{x\xi} & H_{\xi\xi}
  \end{pmatrix}
  \quad\text{is positive semi-definite for all $(x,\xi)$},
\end{align}
and the condition \eqref{eq:elliptic2} is equivalent to
\begin{align}
  \label{eq:elliptic}
  c\le\xi^2H_{\xi\xi}(x,\xi)\le C \quad \text{and} \quad |H_x(x,\xi)-\xi H_{x\xi}(x,\xi)|\leq C
  \quad \text{for all $x\in\setR,\,\xi\in\setR_+$}.
\end{align}
\begin{proof}
  We need to show that the EVI
  \begin{align}
    \label{eq:evi2}
    \frac12\frac{\dd}{\dd t}\wass(\flow_\Psi^t\eta^0,\tilde\eta)^2 + \frac\kappa2\wass(\flow_\Psi^t\eta^0,\tilde\eta)^2 \le \Psi(\tilde\eta)-\Psi(\flow_\Psi^t\eta^0)
  \end{align}
  holds for all smooth and strictly positive densities $\eta^0,\tilde\eta\in\dens$ at almost every time $t>0$.

  Instead of proving \eqref{eq:evi2} directly for $\flow_\Psi$, we prove it for a family of regularizations and then pass to the limit.
  For every $N\in\setN$, the regularized functional $\Psi_N$ is defined by $\Psi_N(\eta)=\Psi(\eta)$ if $\eta\in\dens$ is supported in $[-N,N]$,
  and $\Psi_N(\eta)=+\infty$ otherwise.
  We show that the solution operator $\flow_N$ to the boundary value problem
  \begin{align}
    \label{eq:evolreg}
    \partial_t \eta = \dff_x( \eta F_{\eta\eta}(x,\eta)\eta_x + \eta F_{\eta x}(x,\eta)),
    \quad \eta_x(t,N)=\eta_x(t,-N)=0
  \end{align}
  is a $\kappa$-flow for $\Psi_N$, i.e., the analogue of \eqref{eq:evi2} holds.
  The densities $\tilde\eta$ and $\eta^0$ are approximated by
  \begin{align*}
    \tilde\eta_N = (\eta +\tilde\mu_N)\indy_{[-N,N]}, \quad \eta^0_N = (\eta^0+\mu^0_N)\indy_{[-N,N]},
  \end{align*}
  with constants $\tilde\mu_N,\,\mu^0_N>0$ such that $\tilde\eta_N,\,\eta^0_N\in\dens$.
  Notice that $\Psi_N(\tilde\eta_N)$, $\Psi_N(\eta^0_N)$ converge to $\Psi(\tilde\eta)$, $\Psi(\eta^0)$, respectively, as $N\to\infty$.

  By the lower bound on $\eta F_{\eta\eta}$ required in \eqref{eq:elliptic},
  the equation \eqref{eq:evolreg} is uniformly parabolic.
  Thus the boundary value problem with initial condition $\eta^0_N$ possesses a solution $\eta$ with $\eta(t)=\flow_N^t(\eta^0_N)$ for all $t\ge0$
  such that each $\eta(t)\in\prb$ restricts to a smooth and strictly positive function on $[-N,N]$ and vanishes outside of that interval.
  The associated distribution function $U:[0,\infty[\times[-N,N]\to[0,1]$ with
  \begin{align*}
    U(t;x) = \int_{-\infty}^x \eta(t;y)\dd y
  \end{align*}
  is a smooth diffeomorphism from $[-N,N]$ to $[0,1]$, for every $t\in\setR_+$;
  this follows from smoothness, strict positivity and mass preservation of $\eta$.
  The inverse $G:[0,\infty[\times[0,1]\to[-N,N]$ of $U$ with respect to $x$ satisfies by definition
  \begin{align}
    \label{eq:idf}
    G(t;U(t;x))=x \quad \text{for all $x\in[-N,N]$.}
  \end{align}
  In particular, we have
  \begin{align}
    \label{eq:idfboundary}
    G(t;0)=-N \text{ and } G(t;1)=N.
  \end{align}
  We differentiate equation \eqref{eq:idf} with respect to $x$ and with respect to $t$, respectively,
  to obtain
  \begin{align}
    \label{eq:space}
    1 &= \eta(t;x)\,\partial_\omega G(t;U(t;x)),
    \quad \text{and} \\
    \label{eq:time}
    0 &= \partial_t G(t;U(t;x)) + \partial_t U(t;x)\,\partial_\omega G(t;U(t;x)).
  \end{align}
  On the other hand, integration of \eqref{eq:evol} with respect to $x$ yields
  \begin{align*}
    \partial_t U(t;x) = \eta(t;x)\dff_x\big[F_\eta\big(x,\eta(t;x)\big)\big].
  \end{align*}
  We multiply by $\partial_\omega G(t;U(t;x))$, substitute \eqref{eq:time} for $\partial_tU$ and \eqref{eq:space} for $\eta$.
  This yields
  \begin{align}
    \label{eq:newevol}
    \partial_t G(t;U(t;x)) = -\dff_x\big[F_\eta\big(x,\eta(t;x)\big)\big].
  \end{align}
  We rewrite this (omitting the dependence on $t$) using again \eqref{eq:space} and the properties of $H$ from \eqref{eq:hprop}:
  \begin{align*}
    -\dff_x[F_\eta(x,\eta(x))]
    &= \frac1{\eta(x)} \big(\dff_x[F(x,\eta(x))-\eta(x)F_\eta(x,\eta(x))] - \dff_xF(x,\eta(x))+\partial_x\eta(x)F_\eta(x,\eta(x))\big) \\
    &= \frac1{\eta(x)} \dff_x[H_\xi(x,1/\eta(x))] - H_x(x,1/\eta(x)) \\
    &= \partial_\omega G(U(x)) \,\dff_x\big[H_\xi\big(G(U(x)),G_\omega(U(x))\big)\big] - H_x\big(G(U(x)),\partial_\omega G(U(x))\big).
  \end{align*}
  In combination with \eqref{eq:newevol}, we have
  \begin{align}
    \label{eq:evol2}
    \partial_t G = \dff_\omega[H_\xi(G,\partial_\omega G)] - H_x(G,\partial_\omega G).
  \end{align}
  Further, observe that a change of variables $\omega=U(x)$ leads to
  \begin{align}
    \nonumber
    \Psi_N(\eta) &= \int_{-N}^N F(x,\eta(x))\dd x = \int_{-N}^N F\Big(G(U(x)),\frac1{\partial_\omega G(U(x))}\Big)\dd x  \\
    &= \int_0^1 \partial_\omega G(\omega) F\Big(G(\omega),\frac1{\partial_\omega G(\omega)}\Big)\dd\omega
    \label{eq:newpsi}
    = \int_0^1 H(G(\omega),\partial_\omega G(\omega))\dd\omega.
  \end{align}
  Now, let $\tilde G$ be the inverse distribution function of $\tilde\eta_N$.
  Using the representation \eqref{eq:wassbyidf} of the Wasserstein distance, we obtain
  \begin{align}
    \label{eq:wass}
    \wass\big(\eta(t),\tilde\eta\big)^2 = \int_0^1 [G(t;\omega)-\tilde G(\omega)]^2\dd\omega.
  \end{align}
  Now we can prove \eqref{eq:evi2}.
  Combining \eqref{eq:wass} with the evolution equation \eqref{eq:evol2} and the representation \eqref{eq:newpsi} of $\Psi$,
  we find
  \begin{align*}
    &\frac12\frac{\dd}{\dd t} \wass(\eta(t),\tilde{\eta}_N)^2 + \frac\kappa2 \wass(\eta(t),\tilde{\eta}_N)^2
    = \int_0^1 G_t[G-\tilde{G}]\dd\omega + \frac\kappa2 \int_0^1 [G-\tilde{G}]^2\dd\omega \\
    &= \int_0^1 \dff_\omega[H_\xi(G,G_\omega)][G-\tilde{G}]\dd\omega - \int_0^1 H_x(G,G_\omega)[G-\tilde{G}]\dd\omega + \frac\kappa2 \int_0^1 [G-\tilde{G}]^2\dd\omega \\
    &\stackrel{(\star)}{=} \int_0^1 H_\xi(G,G_\omega)[\tilde{G}_\omega-G_\omega]\dd\omega + \int_0^1 H_x(G,G_\omega)[\tilde{G}-G]\dd\omega + \frac\kappa2 \int_0^1 [\tilde{G}-G]^2\dd\omega \\
    &\stackrel{(\star\star)}{\le} \int_0^1 \big[H(\tilde{G},\tilde{G}_\omega) - H(G,G_\omega) \big]\dd\omega
    = \Psi_N(\tilde{\eta}_N)-\Psi_N(\eta(t)).
  \end{align*}
  In step $(\star\star)$ we have used the joint convexity of \eqref{eq:joint}.
  The equality $(\star)$ is established integrating by parts;
  the boundary terms vanish because of the boundary conditions \eqref{eq:idfboundary} satisfied by $G$ and $\tilde G$.
  This proves the analogue of \eqref{eq:evi2} for each $\Psi_N$ with its respective flow $\flow_N$.

  Finally, we pass to the limit $N\to\infty$.
  For definiteness of notation, we denote the solutions to \eqref{eq:evolreg} by $\eta_N(t)=\flow_N^t(\eta^0_N)$ from now on.

  To begin with, we show that a subsequence $(\eta_{N'})\subseteq(\eta_N)$ converges to a limit function $\eta_\infty$
  satisfying the unregularized evolution equation \eqref{eq:evol}.
  For the following calculations, the bounds required in \eqref{eq:elliptic2} are important.
  First, we derive an $H^1(\setR)$-estimate:
  \begin{align*}
    \frac12\frac{\dd}{\dd t}\int_{-N}^N \eta_N^2\dd x
    &= -\int_{-N}^N \eta_N F_{\eta\eta}(\eta_N,x)\left(\eta_{N}\right)_{x}^2 \dd x - \int_{-N}^N \left(\eta_N\right)_x F_{x\eta}(\eta_N,x)\eta_N \dd x \\
    &  \leq - c \int_{-N}^N (\eta_N)_x^2 \dd x -\frac{c}2\int_{-N}^N (\eta_N)_x^2 \dd x + \frac{C^2}{2c} \int_{-L}^L \eta_N^2 \dd x.
  \end{align*}
  Integration with respect to $t\in]0,T[$ yields an $N$-uniform bound on $\eta_N$ in $L^2(0,T;H^1(\R))$ in terms of $\|\eta^0\|_{L^2}$.
  By Alaoglu's theorem, this provides weak convergence of a subsequence $\eta_{N'}$ to a limit $\eta_\infty$ in $L^2(0,\infty;H^1(\setR))$.
  Second, we prove an $H^{-1}(\setR)$-estimate on $\partial_t\eta_N$.
  To this end, let $\theta_{xx}=\partial_t(\eta_N)$ with $\theta_x(t,N)=\theta_x(t,-N)=0$,
  and calculate, using again \eqref{eq:elliptic2}:
  \begin{align*}
    \int_{-N}^N \theta_x^2 \dd x
    &= -\int_{-N}^N \theta \theta_{xx}\dd x = -\int_{-N}^N \theta\partial_t(\eta_N)\dd x
    = - \int_{-N}^N \theta_x\,(\eta_N)_x\,\eta_N F_{\eta\eta} \dd x + \int_{-N}^N \theta_x\,\eta_N F_{\eta x} \dd x \\
    &  \leq \frac12\int_{-N}^N \theta_x^2 \dd x + C^2\int_{-N}^N(\eta_N)_x^2 \dd x + C^2\int_{-N}^N \eta_N^2 \dd x,
  \end{align*}
  which, in combination with the $H^1(\setR)$-bound from above,
  shows that $\partial_t\eta_N$ is $N$-uniformly bounded in $L^2(0,T;H^{-1}(\R))$.
  The Aubin-Lions lemma, see e.g. \cite{showalter}, shows
  that the convergence of $\eta_{N'}$ to $\eta_\infty$ is actually strong in $L^2_\loc(\setR_+\times \R)$.

  Strong convergence in $L^2_\loc(\setR_+\times\setR)$ and weak convergence in $L^2(0,\infty;H^1(\setR))$ allows us
  to pass to the limit in the weak formulation of \eqref{eq:evolreg},
  \begin{align*}
    -\int_0^\infty\intX \varphi_t\eta_{N'}\dd x\dd t
    = - \int_0^\infty\intX \varphi_x \big[\eta_{N'} F_{\eta\eta}(x,\eta_{N'})(\eta_{N'})_x + \eta_{N'} F_{x\eta}(x,\eta_{N'})\big]\dd x\dd t
  \end{align*}
  where $\varphi\in C^\infty_c(\setR_+\times\setR)$ is a test function.
  Here we use that the nonlinearities $\eta F_{\eta\eta}(x,\eta)$ and $\eta F_{x\eta}(x,\eta)$ grow at most linearly with respect to $\eta$ by assumption \eqref{eq:elliptic2},
  and so they converge in $L^2_\loc(\setR_+\times\setR)$, too, by dominated convergence.
  We conclude that the limit $\eta_\infty$ is a weak (and thus --- by parabolic regularity theory --- also the unique classical) solution
  to the Cauchy problem \eqref{eq:evol} with initial condition $\eta^0$.
  In other words, we have $\flow_\Psi^t(\eta_0)=\eta_\infty(t)$ for almost every $t\ge0$,
  where $\flow_\Psi$ is the solution operator to \eqref{eq:evol}.

  Passage to the limit in the EVI \eqref{eq:evi2} is more subtle
  because the convergence of the time derivatives $\partial_t\eta_{N'}$ is a priori only weak in $L^2(0,\infty;H^{-1}(\setR))$.
  In \cite{daneri}, an equivalent time-integrated version of \eqref{eq:evi2} is given,
  and the latter is stable under strong $L^2_\loc$-convergence.
  We refer the interested reader to \cite[Section 3]{daneri} for further details.
\end{proof}

Lemma \ref{lem:z} gives an indication that the functional $\poten$ from \eqref{eq:poten} is \emph{not} geodesically $\lambda$-convex.
Note that, by Theorem \ref{thm:convexity}, convexity of $\poten$ would imply the existence of a $\kappa$-flow for $\Psi$.
However, with $F(x,\eta)=a(x)\eta^m$, we have $H(x,\xi)=a(x)\xi^{1-m}$, and thus
\begin{align*}
  \dff^2 H(x,\xi) =
  \begin{pmatrix}
    a''(x)\xi^{1-m}-\kappa & -(m-1)a'(x)\xi^{-m}\\ -(m-1)a'(x)\xi^{-m} & m(m-1)a(x)\xi^{-(m+1)}
  \end{pmatrix}.
\end{align*}
Unless $a$ is a constant, there must exist an $\bar x\in\setR$ with $a''(\bar x)<0$.
Thus, no matter how $\kappa$ is chosen, there exists further a sufficiently small $\bar\xi\in\setR_+$
such that the top left element of $\dff^2 H(\bar x,\bar \xi)$ is negative, that is $a''(\bar x)\bar\xi^{-(m-1)}-\kappa<0$.
Consequently, $\dff^2 H(\bar x,\bar \xi)$ is not positive semi-definite.
Note, however, that we have only shown sufficiency of the joint convexity of \eqref{eq:joint} for the existence of a $\kappa$-flow.

\section{Construction of entropy solutions}\label{sec:time_discrete}


%
%
This section is devoted to the proof of Theorem \ref{thm:main}.
\smallskip

Throughout this section, $\rho_\tau^n$ denotes the $n$th iterate in the JKO scheme for given time step $\tau>0$
as described in Section \ref{subsec:JKO}, see also \eqref{eq:mm} below.
Further, $\bar\rho_\tau:[0,\infty[\to\dens$ denotes the piecewise constant interpolation of the sequence $(\rho_\tau^n)_{n\in\setN}$,
with \eqref{eq:interpol}.

\subsection{Well posedness of the JKO scheme}\label{subsec:wellposed}
Recall the inductive definition of the $\rho_\tau^n$,
\begin{align}
  \label{eq:mm}
  \rho_\tau^n\in\dens \quad \text{is the unique minimizer of} \quad \poten_\tau(\rho;\rho_\tau^{n-1}) =  \frac1{2\tau}\wass(\rho,\rho_\tau^{n-1})^2 + \poten(\rho).
\end{align}
\begin{lem}\label{lem:scheme_well_posed}
  The time-discrete scheme is well-defined in the sense that for any given $\rho_\tau^{n-1}\in\dens$,
  the functional $\poten_\tau(\cdot;\rho_\tau^{n-1})$ indeed admits a unique minimizer $\rho_\tau^n\in\dens$.
\end{lem}
\begin{proof}
  We argue that the direct methods from the calculus of variations apply to the minimization problem for $\poten_\tau(\cdot;\rho_\tau^{n-1})$ on $\dens$.

  First, observe that $\rho_\tau^{n-1}\in\dens$ implies that the $\poten(\rho_\tau^{n-1})$-sublevel of this functional is non-empty and weakly-$\star$ relatively compact.
  Indeed, $\rho_\tau^{n-1}$ itself belongs to that sublevel, and if $\rho^*\in\overline{\dens}$ is any other element of this sublevel,
  then $\wass(\rho^*,\rho_\tau^{n-1})^2\le\poten(\rho_\tau^{n-1})$, since $\poten$ is a non-negative functional.
  This bound on the Wasserstein distance provides a control on the second moment of $\rho^*$ in terms of $\rho_\tau^{n-1}$ alone,
  and so the sublevel is a tight collection of measures, and thus relatively compact in the weak-$\star$ topology.
  Further, both $\poten$ and the distance $\wass(\cdot,\rho_\tau^{n-1})$ are lower semi-continuous with respect to narrow convergence.
  This, together with the fact that $\poten_\tau(\cdot;\rho_\tau^{n-1})$ is bounded from below, guarantees the existence of at least one minimizer $\rho_\tau^n\in\overline{\dens}$.
  Finally, since necessarily $\poten(\rho_\tau^n)$ is finite, $\rho_\tau^n$ belongs to the subspace of densities $\dens$.

  To prove uniqueness of the minimizer, it suffices to observe
  that the squared distance $\wass(\cdot,\rho_\tau^{n-1})^2$ is convex,
  and that the potential $\poten$ is \emph{strictly} convex (recall that $m>1$) in the sense of linear interpolation of measures.
  Thus, there can be at most one minimizer.
\end{proof}

\subsection{Basic energy estimates and flow interchange}\label{subsec:basic}
Immediately from the construction in \eqref{eq:mm}, we obtain the canonical energy estimate
\begin{align*}
  \frac1{2\tau} \sum_{n=1}^N\wass(\rho_\tau^n,\rho_\tau^{n-1})^2
  \le \mathcal{F}(\rho_\tau^0) - \mathcal{F}(\rho_\tau^N) \quad \text{for all $N\in\setN$}.
\end{align*}
which induces the following uniform bound on $\rho_\tau^n$ in $L^m(\setR)$:
\begin{align}
  \label{eq:Lmbound}
  \frac{\mina}{m} \big\|\rho_\tau^n\big\|_{L^m}^m
  \le \mathcal{F}(\rho_\tau^n) \le \mathcal{F}(\rho^0).
\end{align}
Since $\mathcal{F}$ is non-negative, it follows further that
\begin{align}
  \label{eq:ee}
  \sum_{n=1}^\infty\wass(\rho_\tau^n,\rho_\tau^{n-1})^2
  \le 2\tau \mathcal{F}(\rho_\tau^0).
\end{align}
For our purposes, we derive stronger a priori estimates by variations of the minimizers $\rho_\tau^n$ along specific $\kappa$-flows,
following the general strategy from \cite[Section 3]{matthes_mccann_savare}.
\begin{lem}
  \label{lem:fi}
  Let $\Psi:\dens\to]-\infty,+\infty]$ be a lower semi-continuous functional on $\dens$ which possesses a $\kappa$-flow $\flow_\Psi$.
  Define further the dissipation of $\poten$ along $\flow_\Psi$ by
  \begin{align*}
    \dissip_\Psi(\rho)
    := \limsup_{s\downarrow 0}\frac1s \big[ \poten(\rho) - \poten\big(\flow_\Psi^s\rho\big) \big]
  \end{align*}
  for every $\rho\in\dens$.
  If $\rho_\tau^{n-1}$ and $\rho_\tau^n$ are two consecutive steps of the minimizing movement scheme \eqref{eq:mm},
  then
  \begin{align}
    \label{eq:fi-original}
    \Psi(\rho_\tau^{n-1})-\Psi(\rho_\tau^n) \ge \tau\dissip_\Psi(\rho_\tau^n) +  \frac\kappa2\wass(\rho_\tau^n,\rho_\tau^{n-1})^2.
  \end{align}
  In particular, $\Psi(\rho_\tau^{n-1})<\infty$ implies $\dissip_\Psi(\rho_\tau^n)<\infty$.
\end{lem}
%
%
\begin{proof}
  Since \eqref{eq:fi-original} is trivial for $\Psi(\rho_\tau^{n-1})=+\infty$,
  there is no loss in generality to assume $\Psi(\rho_\tau^{n-1})<\infty$.
  We can thus use the EVI inequality \eqref{eq:evi} with $\rho:=\rho_\tau^n$ for the comparison measure $\tilde\rho:=\rho_\tau^{n-1}$.
  By lower semi-continuity of $\Psi$, we conclude
  \begin{align}
    \Psi(\rho_\tau^{n-1}) - \Psi(\rho_\tau^n)
    & \ge \limsup_{s\downarrow0}\big( \Psi(\rho_\tau^{n-1}) - \Psi(\flow_\Psi^s\rho_\tau^n)\big) \\
    \label{eq:fi-proof1}
    & \ge \frac12 \limsup_{s\downarrow0} \Big( \frac{\dd^+}{\dd\sigma}\Big|_{\sigma=s} \wass\big( \flow_\Psi^\sigma\rho_\tau^n,\rho_\tau^{n-1} \big)^2 \Big)
    + \frac\kappa2\wass(\rho_\tau^n,\rho_\tau^{n-1})^2.
  \end{align}
  Notice that we have used the (absolute) $\wass$-continuity of the curve $s\mapsto\flow_\Psi^s\rho_\tau^n$ at $s=0$ in the first step.
  The absolute continuity implies further that
  \begin{align}
    \label{eq:fi-proof2}
    \limsup_{s\downarrow0} \Big( \frac{\dd^+}{\dd\sigma}\Big|_{\sigma=s} \wass\big( \flow_\Psi^\sigma\rho_\tau^n,\rho_\tau^{n-1} \big)^2 \Big)
    \ge \limsup_{s\downarrow0} \frac1s \big( \wass(\flow_\Psi^s\rho_\tau^n,\rho_\tau^{n-1})^2 - \wass(\rho_\tau^n,\rho_\tau^{n-1})^2\big).
  \end{align}
  By definition of $\rho_\tau^n$ as a minimizer of $\poten_\tau$ from \eqref{eq:mm},
  we have
  \begin{align*}
    \poten_\tau(\flow_\Psi^s\rho_\tau^n;\rho_\tau^{n-1}) \ge \poten_\tau(\rho_\tau^n;\rho_\tau^{n-1})
  \end{align*}
  for every $s\ge0$,
  and thus
  \begin{align}
    \label{eq:fi-proof3}
    \wass(\flow_\Psi^s\rho_\tau^n,\rho_\tau^{n-1})^2 - \wass(\rho_\tau^n,\rho_\tau^{n-1})^2
    \ge 2\tau \big[ \poten(\rho_\tau^n) - \poten(\flow_\Psi^s\rho_\tau^n)\big].
  \end{align}
  Inserting \eqref{eq:fi-proof3} into \eqref{eq:fi-proof2}, and then \eqref{eq:fi-proof2} into \eqref{eq:fi-proof1} yields \eqref{eq:fi-original}.
\end{proof}
\begin{cor}
  \label{cor:fi-cond}
  Under the hypotheses of Lemma \ref{lem:fi}, let the $\kappa$-flow $\flow_\Psi$ be such that for every $n\in\setN$, 
  the curve $s\mapsto\flow_\Psi^s\rho_\tau^n$ lies in $L^m(\setR)$, where it is differentiable for $s>0$ and continuous at $s=0$.
  And moreover, let a functional $\Kbound:\dens\to]-\infty,\infty]$ satisfy
  \begin{align}
    \label{eq:fi-cond}
    \liminf_{s\downarrow0}\Big(-\frac{\dd}{\dd\sigma}\Big|_{\sigma=s}\poten(\flow_\Psi^\sigma\rho_\tau^n)\Big)
    \ge\Kbound(\rho_\tau^n).
  \end{align}
  Then the following two estimates hold.
  \begin{align}
    \label{eq:fi-1}
    \text{For every $n\in\setN$:}&\quad
    \Psi(\rho_\tau^{n-1})-\Psi(\rho_\tau^n)
    \ge \tau\Kbound(\rho_\tau^n) + \frac\kappa2\wass(\rho_\tau^n,\rho_\tau^{n-1})^2; \\
    \label{eq:fi-2}
    \text{for every $N\in\setN$:}&\quad
    \Psi(\rho_\tau^N) \le \Psi(\rho^0) -\tau\sum_{n=1}^N \Kbound(\rho_\tau^n) + \tau\max(0,-\kappa)\poten(\rho^0).
  \end{align}
\end{cor}
\begin{proof}
  First, we estimate $\dissip_\Psi$ from below by $\Kbound$.
  The $L^m(\setR)$-regularity assumptions on $\flow^\Psi$ imply that $s\mapsto\poten(\flow^\Psi_s\rho_\tau^n)$
  is differentiable for $s>0$ and continuous at $s=0$.
  Thus, by the fundamental theorem of calculus,
  \begin{align*}
    \dissip^\Psi(\rho_\tau^n)
    = \limsup_{\bar s\downarrow0}\frac1{\bar s}\big[\poten(\rho_\tau^n) - \poten(\flow^\Psi_{\bar s}\rho_\tau^n)\big]
    &= \limsup_{\bar s\downarrow0}\int_0^1 \Big(-\frac{\dd}{\dd\sigma}\Big|_{\sigma=\bar sz}\poten(\flow^\Psi_\sigma\rho_\tau^n)\Big)\dd z \\
    &\ge \int_0^1 \liminf_{\bar s\downarrow0}\Big(-\frac{\dd}{\dd\sigma}\Big|_{\sigma=\bar sz}\poten(\flow^\Psi_\sigma\rho_\tau^n)\Big)\dd z
    \ge \Kbound(\rho_\tau^n),
  \end{align*}
  by Fatou's lemma and \eqref{eq:fi-cond}.
  Now estimate \eqref{eq:fi-1} follows directly from \eqref{eq:fi-original}.
  Estimate \eqref{eq:fi-2} is obtained by arranging \eqref{eq:fi-1} in a telescopic sum and combining it with the energy estimate \eqref{eq:ee}.
\end{proof}



\subsection{Refined a priori estimates}
The main ingredient in order to get suitable compactness is the following a priori estimate.
\begin{lem}
  \label{lem:bound1}
  There is a constant $A$ depending only on $\rho^0$ (and in particular not on $\tau$)
  such that the piecewise constant interpolants $\bar\rho_\tau$ satisfy
  \begin{align}
    \label{eq:bound1}
    \big\|\bar\rho_\tau^{m/2}\big\|_{L^2(0,T;H^1(\setR))} \le A(1+T) \quad \text{for all $T>0$}.
  \end{align}
  In particular, $\bar\rho_\tau(t)^{m/2}\in H^1(\setR)$ for every $t>0$.
\end{lem}
For the proof of Lemma \ref{lem:bound1}, we need a $\tau$-uniform control on the second moment of $\bar\rho_\tau$,
which is established in a further lemma.
\begin{lem}
  \label{lem:bound2nd}
  There is a constant $B$ depending only on $\rho^0$ such that
  \begin{align}
    \label{eq:bound2nd}
    \int_\setR |x|^2\bar\rho_\tau(T,x)\dd x \le B(1+T) \quad \text{for all $T>0$}.
  \end{align}
\end{lem}
The proof of Lemma \ref{lem:bound2nd} is obtained by a straight-forward application of the ``flow interchange'' Lemma \ref{lem:fi}.
Specifically, we choose as auxiliary functional $\Psi:=\mom$ the second moment,
\begin{align*}
  \mom(\rho) = \intX x^2\rho(x)\dd x.
\end{align*}
This functional possesses a $\kappa$-flow (with $\kappa=2$), which is explicitly given by
\begin{align*}
  \big(\flow_\mom^s\eta\big)(x) = e^{2s}\eta(e^{2s}x).
\end{align*}
Clearly, $s\mapsto \flow_\mom^s\eta$ defines a curve  in $L^m(\setR)$ for every $\eta\in L^m(\setR)$;
this curve is differentiable for every $s\ge0$.
\begin{proof}[Proof of Lemma \ref{lem:bound2nd}]
  In order to obtain a sensible bound $\Kbound_\mom$ in \eqref{eq:fi-cond},
  we compute the $s$-derivative of $\poten$ along $\flow_\mom^s$ for a given $\eta_0\in\dens\cap L^m(\setR)$:
  \begin{align*}
    \frac{\dd}{\dd s}\Big|_{s=0} \poten\big(\flow_\Psi^s \eta_0)
    &= \frac{\dd}{\dd s}\Big|_{s=0} \bigg( e^{2ms} \intX \frac{a(x)}{m}\eta_0(e^{2s}x)\dd x \bigg)
    = \frac{\dd}{\dd s}\Big|_{s=0} \bigg( e^{2(m-1)s} \intX \frac{a(e^{-2s}z)}{m}\eta_0(z)\dd z\bigg) \\
    &= 2(m-1)\intX \frac{a(z)}{m}\eta_0(z)\dd z - 2\intX \frac{za'(z)}{m}\eta_0(z\dd z
    \le M\poten(\eta_0)
  \end{align*}
  with the constant
  \begin{align*}
    M:= 2(m-1) - 2\inf_{z\in\setR}\frac{za'(z)}{a(z)} \ge 0.
  \end{align*}
  We apply Corollary \ref{cor:fi-cond} with $\Kbound_\mom:=-M\poten$;
  estimate \eqref{eq:fi-2} becomes
  \begin{align*}
    \intX x^2\rho_\tau^n(x)\dd x
    \le \intX x^2\rho^0(x)\dd x + \tau\sum_{n=1}^N\Kbound_\mom(\rho_\tau^n)
    \le \intX x^2\rho^0(x)\dd x + M\poten(\rho^0)\,N\tau,
  \end{align*}
  using the monotonicity \eqref{eq:Lmbound} of the potential.
  From here, \eqref{eq:bound2nd} follows immediately.
\end{proof}
We now turn to the proof of Lemma \ref{lem:bound1} above.
Also here, the key element is an application of Lemma \ref{lem:fi}.
This time, the auxiliary functional $\Psi:=\ent$ is the \emph{entropy},
\begin{align*}
  \ent(\rho) = \intX \rho(x)\log\rho(x)\dd x.
\end{align*}
One of the celebrated results from the theory of optimal transportation is
that this functional possesses a $\kappa$-flow $\flow_\ent$, with $\kappa=0$,
which is given by the heat semi-group, see e.g. \cite{daneri,JKO,villani_book}.
More precisely:
for given $\eta_0\in\dens$, the curve $s\mapsto\eta(s):=\flow_\ent^s\eta_0$ solves the initial value problem
\begin{align}
  \label{eq:heat}
  \partial_s\eta = \eta_{xx} , \quad \eta(0)=\eta_0,
\end{align}
in the classical sense:
$\eta(s)$ is a positive density function for every $s>0$,
it is continuously differentiable as a map from $\setR_+$ to $C^\infty(\setR)\cap L^1(\setR)$,
and if $\eta_0\in L^m(\setR)$, then $\eta(s)$ converges to $\eta_0$ in $L^m(\setR)$ as $s\downarrow0$.

In order to extract information from the flow interchange estimate for $\ent$,
we need another technical ingredient.
\begin{lem}
  \label{lem:boundent}
  There is a constant $C$ depending only on $\rho^0$ such that
  \begin{align}
    \label{eq:boundent}
    - C(1+T) \le \ent(\bar\rho_\tau(t)) \le C \quad \text{for all $T>0$}.
  \end{align}
\end{lem}
\begin{proof}[Proof of Lemma \ref{lem:boundent}]
  One verifies by elementary calculations that
  \begin{align*}
    -\frac2e s^{1/2}<s\log s<\frac1{(m-1)e}s^m\quad \text{for all $s>0$}.
  \end{align*}
  And so, for every $\eta\in\dens$, we have on one hand that
  \begin{align*}
    \ent(\eta) = \intX \eta(x)\log\eta(x)\dd x \le \frac1{(m-1)e}\intX \eta(x)^m\dd x \le \frac{m}{\mina(m-1)e}\poten(\eta),
  \end{align*}
  while on the other hand,
  \begin{align*}
    \ent(\eta) 
    &\ge -\frac2e\intX(1+x^2)^{-1/2}(1+x^2)^{1/2}\eta(x)^{-1/2}\dd x \\
    &\ge -\frac2e\bigg(\intX\frac{\dd x}{1+x^2}\bigg)^{1/2}\bigg(\intX(1+x^2)\eta(x)\dd x\bigg)^{1/2}
    \ge -\frac{2\sqrt\pi}e\bigg( 1 + \intX x^2\eta(x)\dd x\bigg)^{1/2}.
  \end{align*}
  Now, with $\bar\rho_\tau(t)$ in place of $\eta$,
  it is straightforward to conclude both estimates in \eqref{eq:boundent}
  by using the energy bound \eqref{eq:Lmbound} and the moment control \eqref{eq:bound2nd}, respectively.
\end{proof}
\begin{proof}[Proof of Lemma \ref{lem:bound1}]
  By the preceding discussion, $\flow_\ent$ satisfies the hypotheses of Corollary \ref{cor:fi-cond};
  we shall define a suitable lower bound $\Kbound_\ent$ for the use in \eqref{eq:fi-cond}.
  By the spatial regularity of $\eta(s)$ for every $s>0$, the following calculations are justified:
  \begin{align}
    \label{eq:diss_entropy}
    \begin{split}
      \partial_s\poten(\flow_\ent^s\eta_0)
      &= \frac1m\int_\setR a(x)\partial_s\big(\eta(s,x)^m\big)\dd x
      = \int_\setR a(x)\eta(s,x)^{m-1}\eta_{xx}(s,x)\dd x \\
      & = -(m-1) \int_\setR a(x)\eta(s,x)^{m-2}\eta_x(s,x)^2\dd x - \int_\setR a_x(x)\eta(s,x)^{m-1}\eta_x(s,x)\dd x \\
      & = -\frac{4(m-1)}{m^2} \int_\setR a(x) \big[\partial_x\big(\eta(s,x)^{m/2}\big)\big]^2\dd x + \frac1m\int_\setR a_{xx}(x)\eta(s,x)^m\dd x \\
      & \le -\frac{4(m-1)\underline{a}}{m^2} \int_\setR \big[\partial_x\big(\eta(s,x)^{m/2}\big)\big]^2\dd x
      + \frac{\maxaxx}m \int_\setR \eta(s,x)^m\dd x,
    \end{split}
  \end{align}
  where we have used the notation $\maxaxx:=\sup a_{xx}<+\infty$ in view of the property (a2) in Proposition \ref{prop:change_var}.
  Consequently, we define
  \begin{align}
    \label{eq:coercive}
    \Kbound_\ent(\rho) := \frac{4(m-1)\mina}{m^2} \intX \big[\partial_x\big(\rho^{m/2}\big)\big]^2\dd x
    - \frac{\maxaxx}m \intX \rho^m\dd x.
  \end{align}
  It is easily seen that \eqref{eq:fi-cond} is satisfied with this choice of $\Kbound$:
  indeed, since $\flow_\ent^s$ is continuous in $L^m(\setR)$ at $s=0$,
  it suffices to observe that $\Kbound_\ent$ is lower semi-continuous in $L^m(\setR)$ by Lemma \ref{lem:Hlsc}.
  So Corollary \ref{cor:fi-cond} is applicable,
  and from \eqref{eq:fi-2}, we conclude that
  \begin{align}
    \label{eq:entdissipate}
    \ent(\rho_\tau^N) + \frac{4(m-1)\mina}{m^2} \tau\sum_{n=1}^N\big\|\partial_x\big((\rho_\tau^n)^{m/2}\big)\big\|_{L^2}^2
    &\le \ent(\rho^0) + \frac{\maxaxx}m \tau\sum_{n=1}^N\big\| \rho_\tau^n\big\|_{L^m}^m
  \end{align}
  for every $N\in\setN$.
  In combination with the bound \eqref{eq:boundent} on $\ent$, it follows that
  \begin{align*}
    \frac{4(m-1)\mina}{m^2} \tau\sum_{n=1}^N\big\|\big(\rho_\tau^n\big)^{m/2}\big\|_{H^1}^2
    &\le \ent(\rho^0) -\ent(\rho_\tau^N) + \frac{\maxaxx}m \tau\sum_{n=1}^N\big\| \rho_\tau^n\big\|_{L^m}^m
    + \frac{4(m-1)\mina}{m^2} \tau\sum_{n=1}^N\big\|\rho_\tau^n\big\|_{L^m}^m\\
    &\le C(2+N\tau) + \bigg[\frac{\maxaxx}{\mina} + \frac{4(m-1)}{m}\bigg] \poten(\rho^0) N\tau .
  \end{align*}
  The energy estimate \eqref{eq:Lmbound} has been used to derive the last line.
  From here, it is immediate to conclude \eqref{eq:bound1}.
\end{proof}

\subsection{Weak and strong convergence of the discrete curves}
Our first convergence result is a standard consequence of the energy estimate.
\begin{lem}
  \label{lem:veryweak}
  Every vanishing sequence $(\tau_k)_{k\in\setN}$ of time steps $\tau_k>0$ contains a (non-relabelled) subsequence
  such that $\bar\rho_{\tau_k}$ converges --- uniformly on compact time intervals --- in $\wass$
  to a H\"older continuous limit curve $\rho_*:[0,\infty[\to\dens$.
\end{lem}
\begin{proof}
  From \eqref{eq:ee}, we conclude the H\"older-type estimate
  \begin{align}
    \label{eq:holder}
    \wass\big(\bar\rho_\tau(t),\bar\rho_\tau(s)\big) \le \sqrt{2\mathcal{F}(\rho_\tau^0)} \max(\tau,|t-s|)^{1/2}
  \end{align}
  for arbitrary $s,t\ge0$.
  The claim is now obtained
  as a consequence of (a refined version of) the Ascoli-Arzel\'{a} theorem \cite[Proposition 3.3.1]{AGS}.
\end{proof}
In the following, let a sequence $(\tau_k)$ be fixed along which $\bar\rho_{\tau_k}$ converges to some limit $\rho_*$
in the sense described in Lemma \ref{lem:veryweak}.
For the derivation of the entropy formulation \eqref{eq:entropyformulation} for $u_*=\rescale^{-1}[\rho_*]$ in the next section,
the following stronger convergence result is needed.
\begin{prop}[Strong $L^m$-compactness]
  \label{prop:compact}
  The curves $\bar\rho_{\tau_k}:[0,\infty[\to\dens$ converge to $\rho_*$ in $L^m(]0,T[\times\setR)$ for every $T>0$.
\end{prop}
The proof of Proposition \ref{prop:compact} is obtained
via an extension of the Aubin-Lions lemma as given in \cite{rossi_savare}.
The precise statement is recalled here for convenience:
\begin{thm}[Theorem 2 in \cite{rossi_savare}]
  \label{thm:savare}
  On a Banach space $X$, let be given
  \begin{itemize}
  \item a \emph{normal coercive integrand} $\mathfrak{F}:X\to[0,+\infty]$, i.e.,
    $\mathfrak{F}$ is lower semi-continuous and its sublevels are relatively compact in $X$;
  \item a \emph{pseudo-distance} $g:X\times X\to[0,+\infty]$, i.e.,
    $g$ is lower semi-continuous,
    and $g(\rho,\eta)=0$ for any $\rho,\eta\in X$ with $\mathfrak{F}(\rho)<\infty$, $\mathfrak{F}(\eta)<\infty$ implies $\rho=\eta$.
  \end{itemize}
  Let further $U$ be a set of measurable functions $u:]0,T[\to X$, with a fixed $T>0$.
  Under the hypotheses that
  \begin{align}
    \label{eq:savare_hypo}
    \sup_{u\in U}\int_0^T \mathfrak{F}(u(t)) \dd t<\infty
    \quad\text{and}\quad
    \lim_{h\downarrow0}\sup_{u\in U}\int_0^{T-h} g(u(t+h),u(t))\dd t=0,
  \end{align}
  $U$ contains an infinite sequence $(u_n)_{n\in\setN}$ that converges in measure (with respect to $t\in]0,T[$) to a limit $u_*:]0,T[ \to X$.
\end{thm}
\begin{proof}[Proof of Proposition \ref{prop:compact}]
  Fix some $T>0$.
  We verify the hypotheses of Theorem \ref{thm:savare} for a specific choice of $X$, $\mathfrak{F}$, $g$ and $U$.
  First, let $X:=L^m(\setR)$.
  Next, define
  \begin{align*}
    g(\rho,\eta) :=
    \begin{cases}
      \wass(\rho,\eta) & \text{if $\rho,\eta\in\dens$}, \\
      + \infty & \text{otherwise}.
    \end{cases}
  \end{align*}
  Finally, let $\mathfrak{F}$ be given by
  \begin{align*}
    \mathfrak{F}(\rho) =
    \begin{cases}
      \int_\setR \big[\partial_x\big(\rho(x)^{m/2}\big)\big]^2\dd x + \int_\setR x^2\rho(x)\dd x
      & \text{if $\rho\in\dens$ and $\partial_x(\rho^{m/2})\in L^2(\setR)$}, \\
      + \infty & \text{otherwise}.
    \end{cases}
  \end{align*}
  Since any elements $\rho$ and $\eta$ in the proper domain of $\mathfrak{F}$ belong to $\dens$,
  it is clear that $0=g(\rho,\eta)=\wass(\rho,\eta)$ implies $\rho=\eta$.
  Further, the lower semi-continuity of $\mathfrak{F}$ on $L^m(\setR)$ follows from Lemma \ref{lem:Hlsc} in the Appendix \ref{sec:app_LSC}.
  Next we show that --- for any given $c>0$ --- the sublevel $A_c:=\{\rho\in L^m(\setR)|\mathfrak{F}(\rho)\le c\}$ is relatively compact in $L^m(\setR)$.
  To this end, we shall prove below that $B_c:=\{\eta=\rho^{m/2}|\rho\in A_c\}$ is relatively compact in $L^2(\setR)$;
  since the map $\iota:L^2(\setR)\to L^m(\setR)$ with $\iota(\eta)=\eta^{2/m}$ is continuous,
  it then follows that $A_c = \iota(B_c)$ is a relatively compact set in $L^m(\setR)$.
  To show relative compactness of $B_c$ in $L^2(\setR)$,
  we verify the hypotheses of the Frech\'{e}t-Kolmogorov theorem, see e.g. \cite[Theorem IV.8.20]{DunfordSchwartz}.

  \emph{$B_c$ is bounded in $L^2(\setR)$.}
  For any given $q\ge1$, the Gagliardo-Nirenberg interpolation inequality (in combination with H\"older's inequality if $m>2$) provides
  \begin{align}
    \label{eq:gagliardo}
    \int_\setR\eta^{2q}\dd x \le K\bigg(\int_\setR(\partial_x\eta)^2\dd x\bigg)^{\frac{mq-1}{m+1}} \bigg(\int_\setR\eta^{2/m}\dd x\bigg)^{\frac{m(q+1)}{m+1}}.
  \end{align}
  Boundedness of $B_c$ in $L^2(\setR)$ follows directly with the choice $q=1$;
  recall that $\eta^{2/m}=\rho$ lies in the proper domain of $\mathfrak{F}$ and thus has integral one.

  \emph{$B_c$ is tight under translations.}
  For every $\eta\in B_c$ and any $h>0$, we have that
  \begin{align*}
    \int_\setR |\eta(x+h)-\eta(x)|^2\dd x
    = \int_\setR \bigg|\int_0^h \partial_x\eta(x+z)\dd z\bigg|^2\dd x
    \le h \int_0^h \int_\setR \big(\partial_x\eta(x)\big)^2\dd x\dd z
    \le ch^2.
  \end{align*}
  Thus the integral converges to zero uniformly on $B_c$ as $h\downarrow0$.

  \emph{Elements of $B_c$ are uniformly decaying at infinity.}
  For every $\eta\in B_c$ and any $R>0$, we have
  \begin{align*}
    \int_{|x|>R} \eta(x)^2\dd x \le \frac1R \intX |x|\eta(x)^2\dd x
    \le \frac1R \bigg(\intX x^2\eta(x)^{2/m}\dd x\bigg)^{1/2}\bigg(\intX \eta(x)^{4-2/m}\dd x\bigg)^{1/2}.
  \end{align*}
  Inside the last expression, the first integral is less than $c$,
  and also the second one is controlled in terms of $c$, using inequality \eqref{eq:gagliardo} with $q=2-1/m>1$.

  In conclusion, $B_c$ satisfies the hypotheses of the Fr\'{e}chet-Kolmogorov compactness theorem.
  It follows that $\mathfrak F$ has compact sublevels.

  We turn to verify the two hypotheses in \eqref{eq:savare_hypo} for the set $U:=\{\bar\rho_{\tau_k}|k\in\setN\}$.
  The first hypothesis is satisfied because of \eqref{eq:bound1} and \eqref{eq:bound2nd}.
  We establish the second hypotheses as a consequence of the $\tau$-uniform approximate H\"older continuity \eqref{eq:holder}.
  For this, pick $T>0$ and $h\in]0,1[$ arbitrary.
  Given $k\in\setN$, define $N_k\in\setN$ such that $(N_k-1)\tau_k<T\le N_k\tau_k$.
  We distinguish two cases.
  If $0<h<\tau_k$, then (writing $\tau=\tau_k$ and $N=N_k$ for ease of notation)
  \begin{align*}
    \int_0^{T-h} \wass\big(\bar\rho_\tau(t+h),\bar\rho_\tau(t)\big)\dd t
    &\le \tau\sum_{n=0}^{N-1} \frac{h}\tau \wass(\rho_\tau^{n+1},\rho_\tau^n)
    \le h N^{1/2}\bigg(\sum_{\ell=0}^\infty\wass(\rho_\tau^{\ell+1},\rho_\tau^\ell)^2\bigg)^{1/2} \\
    &\le h\big(2 \tau N\poten(\rho_\tau^0)\big)^{1/2}
    \le \big(2(T+1)\poten(\rho_\tau^0)\big)^{1/2}h
  \end{align*}
  by \eqref{eq:ee}.
  If, on the other hand, $h\ge\tau_k$, then there is a $J\in\setN$ with $h\le J\tau_k\le2h$,
  and so
  \begin{align*}
    &\int_0^{T-h} \wass\big(\bar\rho_\tau(t+h),\bar\rho_\tau(t)\big)\dd t
    \le \tau\sum_{n=0}^{N-1} \sum_{j=0}^{J-1} \wass(\rho_\tau^{n+j+1},\rho_\tau^{n+j}) \\
    & \qquad \le \tau \sum_{n=0}^{N-1} \bigg( J^{1/2} \bigg[\sum_{\ell=0}^\infty \wass(\rho_\tau^{\ell+1},\rho_\tau^\ell)^2 \bigg]^{1/2}\bigg)
    \le N\tau \big( 2J\tau\poten(\rho^0) \big)^{1/2} \le 2(T+1)\poten(\rho^0)^{1/2}\,h^{1/2}.
  \end{align*}
  Theorem \ref{thm:savare} now provides for every subsequence $(\tau_{k'})\subseteq(\tau_k)$
  the existence of a subsubsequence $(\tau_{k''})\subseteq(\tau_{k''})$ such that
  $\bar\rho_{\tau_{k''}}$ converges in measure with respect to $t\in]0,T[$ in $L^m(\setR)$ to some limit $\rho_+$.
  By convergence of $\bar\rho_{\tau_k}(t)$ to $\rho_*(t)$ in $\wass$ for every $t\in[0,T]$,
  it follows that $\rho_+=\rho_*$.
  By the usual arguments, we conclude that the entire sequence $(\bar\rho_{\tau_k})$ converges to $\rho_*$ in measure.
  In combination with the $\tau$-uniform bound \eqref{eq:Lmbound} of $\bar\rho_\tau$ in $L^m(\setR)$,
  we can invoke Lebesgue's dominated convergence theorem to conclude strong convergence
  of $\bar\rho_{\tau_k}$ to $\rho_*$ in $L^m(0,T;L^m(\setR))$.
\end{proof}
\begin{cor}
  \label{cor:regularity}
  For every $T>0$, we have $\rho_*^{m/2}\in L^2(0,T;H^1(\setR))$
  and $\partial_x(\rho_*^m)\in L^1(0,T;L^1(\setR))$.
\end{cor}
\begin{proof}
  Fix $T>0$.
  By estimate \eqref{eq:bound1},
  $\bar\rho_{\tau_k}^{m/2}$ is uniformly bounded in the reflexive Banach space $L^2(0,T;H^1(\setR))$.
  By Alaoglu's theorem, there is a subsequence $(\tau_{k'})\subseteq(\tau_{k})$ such that
  $\bar\rho_{\tau_{k'}}^{m/2}$ converges weakly to some limit $\zeta$ in that space.
  Since $\bar\rho_{\tau_{k'}}^{m/2}$ converges strongly to $\rho_*^{m/2}$ in $L^2(0,T;L^2(\setR))$ by Proposition \eqref{prop:compact},
  it follows that $\rho_*^{m/2}=\zeta\in L^2(0,T;H^1(\setR))$.

  The second claim is a trivial consequence of the representation $\partial_x(\rho_*^m)=2\rho_*^{m/2}\partial_x(\rho_*^{m/2})$.
\end{proof}

\subsection{Derivation of the entropy formulation}\label{subsec:derivation_entropy}
In this subsection we show that $u_*=\rescale^{-1}[\rho_*]$ satisfies the entropy formulation \eqref{eq:entropyformulation}.
The following proposition plays the role of Lemma 2.4 in \cite{karlsen2003}.
We emphasize that although the entropy inequality contains the same dissipation term as in \cite{karlsen2003},
we derive it from a completely different source.
Here, it results naturally from the variational construction by minimizing movements.
\begin{prop}
  \label{prop:karlsen}
  Define $u_*=\rescale^{-1}[\rho_*]$ from $\rho_*$ via scaling, see \eqref{eq:defscaling}.
  Then $u_*$ satisfies the entropy inequality \eqref{eq:entropyformulation}
  for any $k>0$ and for any non-negative test function $\varphi\in C^\infty_c(\setR_+\times\setR)$.
\end{prop}
\begin{proof}
  In view of Remark \ref{rmk:tensor}, it suffices to prove the estimate \eqref{eq:entropyformulation}
   for all test functions $\varphi$ of the form $\varphi(t,y)=\theta(t)\phi(y)$,
  with arbitrary non-negative $\theta\in C^\infty_c(\setR_+)$ and $\phi\in C^\infty_c(\setR)$.
  Let $\phi$ and $\theta$ as well as $k>0$ be fixed in the following.

  The proof of the entropy inequality results from another application of Lemma \ref{lem:fi}.
  For given parameters $\nu>0$ and $0<\eps<k$, define
  \begin{align}
    \label{eq:1}
    \Psi_{\eps,\nu}(\eta) = \int_\setR S_\epsilon\Big(\frac{\eta(x)}{\trf'(x)}\Big)\phi\circ \trf(x)\trf'(x)\dd x + \nu\ent(\eta),
  \end{align}
  where the function $S_\eps:\setR\to\setR$ is given by
  \begin{align*}
    S_\eps(s) := \int_0^s\sgn_\eps(r^m-k^m)\dd r.
  \end{align*}
  From its definition, it is obvious that $S_\eps(s)$ is bounded from below by $-k$
  and converges monotonically (from above) and uniformly in $s\in\setR$ to $|s-k|-k$ in the limit $\eps\downarrow0$.
  The first and second derivatives are given by
  \begin{align*}
    S_\eps'(s) = \sgn_\eps(s^m-k^m), \quad S_\eps''(s) = 2ms^{m-1}\delta_\eps(s^m-k^m).
  \end{align*}
  For later reference, we observe that there is a constant $K_\eps$ such that
  \begin{align}
    \label{eq:sproperties}
    |S_\eps(s)|\le s, \quad |S_\eps'(s)| \le 1, \quad 0\le S_\eps''(s)\le K_\eps s^{-1} \quad \text{for all $s\in\setR_+$},
  \end{align}
  which can be verified by elementary calculations, using the definitions of $\delta_\eps$ and $\sgn_\eps$.

  As a preliminary step, we show that there is a $\kappa$-flow associated to $\Psi_{\eps,\nu}$,
  which is given as the solution operator $\flow_{\eps,\nu}$ to
  \begin{align}
    \label{eq:dumb002}
    \partial_s\eta = \Big( \eta \Big[\sgn_\eps\bigg(\Big(\frac{\eta}{\trf'}\Big)^m-k^m\bigg)\phi\circ \trf\Big]_x\Big)_x + \nu\eta_{xx}.
  \end{align}
  This will be achieved by application of Lemma \ref{lem:z}.
  In the situation at hand, we have
  \begin{align*}
    F(x,\eta) = S_\eps\Big(\frac{\eta}{\trf'(x)}\Big)\phi\circ \trf(x)\trf'(x) + \nu\eta\log\eta,
  \end{align*}
  and it is easily seen that \eqref{eq:elliptic2} is satisfied (with $c=\nu>0$).
  The associated function $H$ reads
  \begin{align*}
    H(x,\xi) = \xi S_\eps\Big(\frac1{\xi \trf'(x)}\Big)\phi\circ \trf(x)\trf'(x) - \nu\log\xi.
  \end{align*}
  We calculate the entries of the matrix in \eqref{eq:convex}:
  \begin{align*}
    H_{\xi\xi}(x,\xi) &= \xi^{-3}S_\eps''\Big(\frac1{\xi \trf'(x)}\Big)\frac{\phi\circ \trf(x)}{\trf'(x)} + \nu\xi^{-2}, \\
    H_{xx}(x,\xi) &=  \xi S_\eps\Big(\frac1{\xi \trf'(x)}\Big)f_1(x) + S_\eps'\Big(\frac1{\xi \trf'(x)}\Big)f_2(x) + \xi^{-1}S_\eps''\Big(\frac1{\xi \trf'(x)}\Big)f_3(x), \\
    H_{x\xi}(x,\xi) &= S_\eps\Big(\frac1{\xi \trf'(x)}\Big)f_4(x) + \xi^{-1}S_\eps'\Big(\frac1{\xi \trf'(x)}\Big)f_5(x) + \xi^{-2}S_\eps''\Big(\frac1{\xi \trf'(x)}\Big)f_6(x).
  \end{align*}
  Here $f_1$ to $f_6$ are continuous functions of compact support in $\setR$, explicitly expressible in terms of $\trf$ and its derivatives.
  Using the properties \eqref{eq:sproperties}, it is easily seen that there is a constant $M_\eps$ such that
  \begin{align*}
    H_{\xi\xi}(x,\xi) \ge \nu\xi^{-2}, \quad
    H_{xx}(x,\xi) \ge -M_\eps, \quad
    |H_{x\xi}(x,\xi)| \le M_\eps \xi^{-1}
  \end{align*}
  holds for all $\xi\in\setR_+$ and uniformly in $x\in\setR$.
  Thus, for a suitable choice of $\kappa=\kappa_{\eps,\nu}$,
  the matrix in \eqref{eq:convex} is positive for every $\xi\in\setR_+$ and $x\in\setR$,
  and so $H$ is jointly convex.
  By Lemma \ref{lem:z}, the solution operator $\flow_{\eps,\nu}$ for \eqref{eq:dumb002} is a $\kappa_{\eps,\nu}$-flow for $\Psi_{\eps,\nu}$.

  In order to apply Corollary \ref{cor:fi-cond}, we need to calculate the derivative of $\poten$ along solutions $\eta$ to \eqref{eq:dumb002}.
  For simplification, introduce the rescaling $v(t)=\rescale^{-1}[\eta(t)]$, see \eqref{eq:defscaling}.
  Further, recall the definitions of $\trf$ and of $\alpha=\trf'\circ \trf^{-1}$,
  and properties \eqref{eq:defT} and \eqref{eq:bfroma}.
  Using the functional $\Kbound_\ent$ defined in \eqref{eq:coercive},
  we have
  \begin{align*}
    &\frac{\dd}{\dd s}\poten(\flow_{\eps,\nu}^s\eta_0)
    = - \intX \eta\big[a\eta^{m-1}\big]_x \big[\sgn_\eps\big((\eta/\trf')^m-k^m\big)\phi\circ \trf\big]_x\dd x
    + \nu\intX a\eta^{m-1}\eta_{xx} \\
    &\le - \intX \{\alpha v\}\circ \trf \big[\{(a\circ \trf^{-1})\cdot(\alpha v)^{m-1}\}\circ \trf\big]_x \big[\{\sgn_\eps\big(v^m-k^m\big)\phi\}\circ \trf\big]_x\dd x
    - \nu\Kbound_\ent(\eta) \\
    &\stackrel{\eqref{eq:defT}}{=} - \intX \alpha^2v\Big[\frac{m\alpha^{-2}}{m-1}v^{m-1}\Big]_y\big[\sgn_\eps(v^m-k^m)\phi\big]_y\dd y  - \nu\Kbound_\ent(\eta) \\
    &\stackrel{\eqref{eq:bfroma}}{=} - \intX \big[ (v^m)_y + bv^m\big][\sgn_\eps(v^m-k^m)\phi]_y\dd y  - \nu\Kbound_\ent(\eta) \\
    &= - \intX \big[(v^m)_y+b(v^m-k^m)\big]\sgn_\eps(v^m-k^m)\phi_y\dd y - \intX k^mb\big[\sgn_\eps(v^m-k^m)\phi\big]_y\dd y \\
    & \qquad -\intX \big[(v^m)_y+b(v^m-k^m)\big](v^m)_y\sgn_\eps'(v^m-k^m)\phi\dd y -\nu\Kbound_\ent(\eta) \\
    & = - \intX \phi_{yy} S_\eps(v)\dd y + \intX \big(b(v^m-k^m)\phi_y-k^mb_y\phi\big) \sgn_\eps(v^m-k^m)\dd y \\
    & \qquad  -\intX \big[P_\eps(v)_y\big]^2\phi\dd y - \intX b Q_\eps(v)_y\phi\dd y - \nu\Kbound_\ent(\eta).
    %
  \end{align*}
  In the last step, we have implicitly defined the smooth functions $P_\eps,Q_\eps:\setR_+\to\setR$ such that
  \begin{align*}
    P_\eps'(s)^2 = \sgn_\eps'(s^m-k^m) \quad \text{and} \quad Q_\eps'(s) = (s^m-k^m)\sgn_\eps'(s^m-k^m).
  \end{align*}
  Accordingly, we define, still with $v=\rescale^{-1}[\eta]$,
  \begin{align*}
    \Kbound_{\eps,\nu}(\eta)
    &= - \intX \phi_{yy} S_\eps(v)\dd y
    + \intX \big(b(v^m-k^m)\phi_y-k^mb_y\phi\big) \sgn_\eps(v^m-k^m)\dd y \\
    & \qquad  +\intX \big[P_\eps(v)_y\big]^2\phi\dd y - \intX (b\phi)_y Q_\eps(v)\dd y + \nu\Kbound_\ent(\eta).
  \end{align*}
  Considered as a functional of $v$, the right-hand side is lower semi-continuous with respect to strong convergence of $v$ in $L^m(\setR)$.
  Indeed, all of the integral expressions are even continuous in $L^m(\setR)$,
  except for the one involving $P_\eps$, for which lower semi-continuity can be concluded by means of Lemma \ref{lem:Hlsc}.
  Since convergence of $\eta$ in $L^m(\setR)$ is equivalent to convergence of $v=\rescale^{-1}[\eta]$ in $L^m(\setR)$,
  the functional $\Kbound_{\eps,\nu}$ is lower-semicontinuous with respect to $\eta$.
  Now since $\flow_{\eps,\nu}^s\eta_0$ converges to $\eta_0$ in $L^m(\setR)$ for every $\eta_0\in L^m(\setR)$,
  we conclude that $\Kbound_{\eps,\nu}$ satisfies the condition \eqref{eq:fi-cond} for the application of Corollary \ref{cor:fi-cond}.

  We shall now derive a refined version of estimate \eqref{eq:fi-2}.
  To this end, recall that $\theta\in C^\infty_c(\setR_+)$ is a temporal test function.
  Multiply \eqref{eq:fi-1} by $\theta(n\tau)$ and sum over $n$ to find
  \begin{align*}
    \tau\sum_{n=1}^\infty \Psi_{\eps,\nu}(\rho_\tau^n)\frac{\theta(n\tau)-\theta((n+1)\tau)}{\tau}
    \ge \tau \sum_{n=1}^\infty \Kbound_{\eps,\nu}(\rho_\tau^n)\theta(n\tau) + \kappa_{\eps,\nu}\tau\poten(\rho^0).
  \end{align*}
  Thanks to the strong convergence of $\bar\rho_\tau$ to $\rho_*$ in $L^m(\setR)$ and the lower semi-continuity of $\Kbound_{\eps,\nu}$,
  we can pass to the time-continuous limit and find
  \begin{align*}
    \intT \Psi_{\eps,\nu}(\rho_*(t))\theta'(t)\dd t \ge \intT \theta(t)\Kbound_{\eps,\nu}(\rho_*(t))\dd t .
  \end{align*}
  Taking into account estimate \eqref{eq:boundent} on $\ent(\rho_*(t))$,
  we can now pass to the inviscid limit $\nu\downarrow0$ and find, recalling $\varphi(t,y)=\theta(t)\phi(y)$,
  that
  \begin{align*}
    \intT\intX S_\eps(u_*-k^m)\partial_t\varphi\dd y\dd t
    \ge &\intT\intX \Big(\big[(u_*^m)_y+b(u_*^m-k^m)\big]\varphi_y-k^mb_y\varphi\Big) \sgn_\eps(u_*^m-k^m)\dd y\dd t \\
    & \quad  + \intT\intX \big[P_\eps(u_*)_y\big]^2\varphi\dd y\dd t - \intT\intX (b\varphi)_y Q_\eps(u_*)\dd y\dd t.
  \end{align*}
  Notice the integration by parts in the first term,
  which is admissible since $\partial_x(\rho_*^m)\in L^1(0,T;L^1(\setR))$ by Corollary \ref{cor:regularity}.
  In the final step, we pass to the limit $\eps\downarrow0$.
  By the uniform convergence of $S_\eps(s)$ to $|s-k|-k$, by the uniform convergence of $Q_\eps$ to zero,
  and since
  \begin{align*}
    \big[P_\eps(u_*)_y\big]^2 = \sgn_\eps'(u_*^m-k^m)\big[(u_*^m)_y\big]^2
  \end{align*}
  by definition of $P_\eps$, we finally obtain \eqref{eq:entropyformulation}.
\end{proof}

\subsection{$L^\infty$ bounds}
Definition \ref{def:entropy_sol_1} of entropy solutions requires $u_*\in L^1\cap L^\infty(]0,T[\times\setR)$ for every $T>0$.
The $L^1$-bound is obvious from our construction.
Below, we prove the $L^\infty$-bound.
\begin{prop}
  \label{prop:Linfty}
  Assume that $\rho^0\in\dens\cap L^\infty(\setR)$, and let $k>0$ be such that $\rho^0(x)\le k a(x)^{-1/(m-1)}$ for almost all $x\in\setR$.
  Then $\rho_*(T,x)\le k a(x)^{-1/(m-1)}$ for all $T>0$ and almost every $x\in\setR$.
\end{prop}
\begin{proof}
  Once again, the proof is obtained in application of Lemma \ref{lem:fi}.
  As auxiliary functional, we choose
  \begin{align*}
    \Psi_{\eps,\nu}(\eta) = \intX \heav_\eps\big( a(x)^{\frac1{m-1}}\eta(x)-k\big)a(x)^{\frac{-1}{m-1}}\dd x + \nu \ent(\eta)
  \end{align*}
  with $\nu>0$ and $\eps\in]0,k[$.
  We verify that $\Psi_{\eps,\nu}$ satisfies the assumptions required in Lemma \ref{lem:z}.
  With the short hand notation $A=a^{1/(m-1)}$, we have
  \begin{align*}
    H(x,\xi) = \xi\heav_\eps\Big( \frac{A(x)}\xi-k\Big)A(x)^{-1} - \nu\log\xi.
  \end{align*}
  Recall that $\heav_\eps'=\stp_\eps$ and $\heav_\eps''=\delta_\eps$,
  and observe that there is some $K_{k,\eps}$ such that
  \begin{align}
    \label{eq:somebounds}
    0\le\heav_\eps(s-k)\le s, \quad
    0\le\stp_\eps(s-k)\le1, \quad
    0\le\delta_\eps(s-k)\le K_{k,\eps} s^{-1}
  \end{align}
  for all $s\in\setR_+$.
  We obtain for the second derivatives of $H$:
  \begin{align*}
    H_{\xi\xi}(x,\xi) &= \frac{A(x)}{\xi^{3}}\delta_\eps\Big( \frac{A(x)}\xi-k\Big) + \nu\xi^{-2}, \\
    H_{xx}(x,\xi) &=  f_1(x)\xi\heav_\eps\Big( \frac{A(x)}\xi-k\Big) + f_2(x)\stp_\eps\Big( \frac{A(x)}\xi-k\Big) + \frac{f_3(x)}{\xi}\delta_\eps\Big( \frac{A(x)}\xi-k\Big), \\
    H_{x\xi}(x,\xi) &= f_4(x)\heav_\eps\Big( \frac{A(x)}\xi-k\Big) + \frac{f_5(x)}{\xi}\stp_\eps\Big( \frac{A(x)}\xi-k\Big) + \frac{f_6(x)}{\xi^{2}}\delta_\eps\Big( \frac{A(x)}\xi-k\Big),
  \end{align*}
  where the functions $f_1$ to $f_6$ are explicitly expressible in terms of $A$, $A'$ and $A''$, and are uniformly bounded in $x\in\setR$.
  Taking into account \eqref{eq:somebounds} and that $A\ge \mina^{1/(m-1)}$,
  we arrive at the uniform bounds
  \begin{align*}
    H_{\xi\xi}(x,\xi) \ge \nu\xi^{-2}, \quad
    H_{xx}(x,\xi) \ge -M_\eps, \quad
    |H_{x\xi}(x,\xi)| \le M_\eps \xi^{-1}
  \end{align*}
  with a suitable constant $M_\eps$.
  Consequently, by application of Lemma \ref{lem:z}, the solution operator $\flow_{\Psi}$ to
  \begin{align*}
    \partial_s \eta = \dff_x\big(\eta\dff_x\big[\stp_\eps(a^{\frac1{m-1}}\eta-k)\big]\big) + \nu\eta_{xx}
  \end{align*}
  defines a $\kappa$-flow for $\Psi_{\eps,\nu}$.
  The regularizing effect of the viscous term is strong enough to justify the following calculations:
  \begin{align*}
    \frac{\dd}{\dd s}\poten(\flow_\Psi^s\eta_0)
    &= \intX a\eta^{m-1}\big(\eta\big[\stp_\eps(a^{\frac1{m-1}}\eta-k)\big]_x\big)_x\dd x + \nu\intX a\eta^{m-1}\eta_{xx}\dd x \\
    &\le -\intX \eta \big[ (a^{\frac1{m-1}}\eta)^{m-1} \big]_x\big[\stp_\eps(a^{\frac1{m-1}}\eta-k)\big]_x\dd x
    - \frac{\nu\maxaxx}{m} \intX \eta^m\dd x.
  \end{align*}
  Here we have used that the second integral with $\nu$ can be estimated as in \eqref{eq:diss_entropy},
  neglecting a positive term.
  The product under the first integral is always non-negative
  since both the functions $s\mapsto s^{m-1}$ and $s\mapsto\stp_\epsilon(s-k)$ are differentiable and increasing.
  Thus, condition \eqref{eq:fi-cond} is satisfied with
  \begin{align*}
    \Kbound(\eta) := - \frac{\nu\maxaxx}{m} \intX \eta^m\dd x,
  \end{align*}
  which is obviously continuous with respect to strong convergence in $L^m(\setR)$.
  Estimate \eqref{eq:fi-2} in combination with \eqref{eq:Lmbound} yields
  \begin{align*}
    \Psi_\nu(\rho_\tau^N)
    &\le \Psi_\nu(\rho^0) + \frac{\nu\maxaxx}{m}\tau\sum_{n=1}^N\intX(\rho_\tau^n)^m\dd x + 2\tau(-\kappa)\poten(\rho^0) \\
    & \le \Psi_\nu(\rho^0) + \Big(\frac{\nu\maxaxx}{m\mina}N\tau + 2\tau(-\kappa)\Big)\poten(\rho^0).
  \end{align*}
  For fixed positive parameters $\nu$ and $\eps$, the modulus $\kappa$ of convexity is a $\tau$-independent constant.
  We can thus pass to the limit $\tau_k\downarrow0$ in \eqref{eq:fi-2} and obtain,
  using the lower semi-continuity of $\Psi_{\eps,\nu}$ in $\wass$,
  that
  \begin{align*}
    \Psi_\nu(\rho_*(T)) \le \Psi_\nu(\rho^0) + \frac{\nu\maxaxx}{m\mina}T\poten(\rho^0)
  \end{align*}
  for every $T\ge0$.
  Using further that $\ent(\rho_*(T))$ is a finite quantity, see \eqref{eq:boundent},
  we can pass to the limit $\nu\downarrow0$ and obtain
  \begin{align*}
    \intX \heav_\eps\big(a(x)^{\frac1{m-1}}\rho_*(T,x)-k\big)a(x)^{\frac{-1}{m-1}}\dd x
    \le \intX \heav_\eps\big(a(x)^{\frac1{m-1}}\rho^0(x)-k\big)a(x)^{\frac{-1}{m-1}}\dd x.
  \end{align*}
  By the properties of $a$, and since $\rho^0,\rho_*(t)\in L^m(\setR)$,
  we can further pass to the limit $\eps\downarrow0$,
  which yields
  \begin{align*}
    \intX \big[\rho_*(T,x)-ka(x)^{\frac{-1}{m-1}}\big]_+\dd x
    \le \intX \big[\rho^0(x)-ka(x)^{\frac{-1}{m-1}}\big]_+\dd x.
  \end{align*}
  Since the integral on the right-hand side is zero by hypothesis, so the is the integral on the left-hand side.
  This proves the claim.
\end{proof}
\begin{cor}
  \label{cor:h1}
  Provided that $\rho^0\in L^\infty(\setR)$, it follows that $\rho_*^m\in L^2(0,T;H^1(\setR))$.
\end{cor}
\begin{proof}
  By Proposition \ref{prop:Linfty},
  \begin{align*}
    K:=\sup_{t\in\setR_+}\esssup_{x\in\setR}\rho_*(t,x)<\infty,
  \end{align*}
  and so
  \begin{align*}
    \intX \big[\partial_x(\rho_*(t)^m)\big]^2\dd x
    = \intX \big[2\rho_*(t)^{m/2}\partial_x(\rho_*(t)^{m/2})\big]^2\dd x
    \le 4K^m\intX \big[\partial_x(\rho_*(t)^{m/2})\big]^2\dd x
  \end{align*}
  for all $t\ge0$.
  The claim now follows from Corollary \ref{cor:regularity}.
\end{proof}

\subsection{Continuity at $t=0$}
It remains to verify that $u_*$ attains the initial condition $u^0$.
\begin{prop}\label{prop:right_cont_zero}
  $u_*(t)\to u^0$ in $L^1(\setR)$ as $t\downarrow0$.
\end{prop}
\begin{proof}
  Since the rescaling $\rescale$ from \eqref{eq:defscaling} is a homeomorphism on $L^1(\setR)$,
  it suffices to show that $\rho_*(t)\to\rho^0$ in $L^1(\setR)$ as $t\downarrow0$.
  By the lower semi-continuity of $\poten$, we can pass to the time-continuous limit $\tau\downarrow0$ in \eqref{eq:Lmbound}
  and obtain
  \begin{equation*}
    \poten(\rho_*(t))\leq \liminf_{\tau\downarrow 0} \poten(\bar \rho_\tau(t)) \leq \poten(\rho^0)
  \end{equation*}
  for every $t>0$,
  which implies that
  \begin{equation}\label{eq:Lmzero_limsup}
    \limsup_{t\downarrow 0}\poten(\rho_*(t))\leq \poten(\rho^0).
  \end{equation}
  On the other hand, since the limiting curve $\rho_*$ is continuous in $\wass$,
  once again the lower semi-continuity of $\poten$ in $\wass$ yields
  \begin{equation*}
    \poten(\rho^0)\leq \liminf_{t\downarrow 0} \poten (\rho_*(t)).
  \end{equation*}
  In combination with \eqref{eq:Lmzero_limsup}, we have that
  \begin{equation}
    \label{eq:normconvergence}
    \lim_{t\downarrow 0} \poten(\rho_*(t)) = \poten(\rho^0).
  \end{equation}
  By definition, $\poten(\rho)$ is the $m$th power of the $L^m$-norm of $\rho$ with respect to the non-uniform background measure $m^{-1}a(x)\dd x$.
  Since $m>1$, and since the weight function $a$ satisfies the bounds (a1)\&(a2),
  we can conclude by standard arguments, see e.g. \cite[Theorem 2.11]{lieb-loss},
  that weak convergence of $\rho_*(t)$ and convergence \eqref{eq:normconvergence} together
  imply strong convergence of $\rho_*(t)$ to $\rho^0$ in $L^m(\setR)$.
  To obtain convergence in $L^1(\setR)$, we apply the generalized H\"older inequality with exponents $2m/(m-1)$, $2$ and $2m$:
  \begin{align*}
    \intX |\rho_*(t)-\rho^0|\dd x
    &\le \intX (1+|x|^2)^{-\frac12}\big[(1+|x|^2)|(\rho_*(t)+\rho^0)\big]^{\frac12}\big[|\rho_*(t)-\rho^0|\big]^{\frac12}\dd x \\
    &\le \bigg[\intX (1+|x|^2)^{-\frac{m}{m-1}}\dd x\bigg]^{\frac{m-1}{2m}}\bigg[\intX(1+|x|^2)(\rho_*(t)+\rho^0)\dd x\bigg]^{\frac12}\|\rho_*(t)-\rho^0\|_{L^m}^{\frac1{2m}}.
  \end{align*}
  The first integral on the right hand side is clearly finite,
  and the second integral remains uniformly bounded as $t\downarrow0$,
  since continuity in $\wass$ implies continuity of the second moment.
  The last term vanishes for $t\downarrow0$ because of the strong convergence of $\rho_*(t)$ in $L^m(\setR)$.
\end{proof}

\subsection{Proof of Theorem \ref{thm:main}}
At this point, we have proven that $u_*$ meets all the requirements for being an entropy solution as stated in Definition \ref{def:entropy_sol_1}:
we have $u_*\in L^1\cap L^\infty(\setR)$ by Proposition \ref{prop:Linfty}
and $u_*^m\in L^2(0,T;H^1(\setR))$ by Corollary \ref{cor:h1};
we further have continuity of $u_*$ at $t=0$ by Lemma \ref{prop:right_cont_zero},
and the validity of the entropy inequality has been verified in Proposition \ref{prop:karlsen}.
Finally, convergence of the time-discrete approximation scheme in Wasserstein and in $L^m$ have been shown
in Lemma \ref{lem:veryweak} and Proposition \ref{prop:compact}, respectively.

\section{Uniqueness of entropy solutions}\label{sec:uniqueness}
In this section we prove Theorem \ref{thm:uniqueness}, using the \emph{doubling of the variables} device.
Since we follow almost literally the proof of \cite[Theorem 1.1]{KR},
we restrict ourselves to the key steps and refer the interested reader to the original article \cite{KR} for more details.

For a fixed $T>0$ we shall use the notation $\Pi_T=\R\times ]0,T[$.
Let $\varphi\in C_c^\infty(\Pi_T\times \Pi_T)$ be a non-negative test function,
and assume that $u$ and $v$ are entropy solutions in the sense of Definition \ref{def:entropy_sol_1}.
For brevity, we write $v=v(x,t)$, $u=u(y,s)$ and  $\varphi=\varphi(x,t,y,s)$.
From Proposition \ref{prop:karlsen} we obtain
\begin{equation}
  \label{eq:doubling1}
  \begin{split}
    -&\iint_{\Pi_T} \iint_{\Pi_T}\big(|v-u|\varphi_t +\sgn(v-u)[(v^m-u^m)b(x)-(v^m)_x]\varphi_x \\
    & \qquad -\sgn(v-u)b'(x)u^m \varphi\big)\dd x \dd t \dd y \dd s \\
    & \leq -\limsup_{\eps\downarrow 0}\iint_{\Pi_T} \iint_{\Pi_T}(v^m)_x^2\sgn'_\eps(v^m-u^m)\varphi \dd x \dd t \dd y \dd s.
  \end{split}
\end{equation}
Recall that $\sgn_\eps$ is a smooth uniformly convergent approximation of the sign function,
which is obtained by mollification with $\delta_\eps$.
Since $(v^m)_x\in L^2(\Pi_T)$,
the following integration by parts is justified:
\begin{align}
  \label{eq:doubling3}
  -\iint_{\Pi_T}\sgn_\eps(v^m-u^m)(v^m)_x \varphi_y \dd y \dd s = \iint_{\Pi_T}(\sgn_\eps(v^m-u^m))_y(v^m)_x \varphi \dd y \dd s.
\end{align}
We integrate \eqref{eq:doubling3} w.r.t. $(x,t)\in\Pi_T$ and send $\epsilon\downarrow0$.
By the dominated convergence theorem, that yields
\begin{equation}
  \label{eq:doubling5}
  \begin{split}
    -\iint_{\Pi_T}\iint_{\Pi_T}&\sgn(v-u)(v^m)_x\varphi_y \dd y\dd s\dd x\dd t \\
    &= -\lim_{\eps\downarrow 0}\iint_{\Pi_T}\iint_{\Pi_T}(u^m)_y (v^m)_x \sgn'_\eps(v^m-u^m)\varphi\dd y\dd s\dd x\dd t.
  \end{split}
\end{equation}
Adding \eqref{eq:doubling1} and \eqref{eq:doubling5} we get
\begin{equation}
  \label{eq:doubling7}
  \begin{split}
    - & \iint_{\Pi_T} \iint_{\Pi_T}\big(|v-u|\varphi_t +\sgn(v-u)\big[(v^m-u^m)b(x)\varphi_x -(v^m)_x(\varphi_x + \varphi_y)\big]\\
    & \qquad -\sgn(v-u)b'(x)u^m \varphi\big)\dd x \dd t \dd y \dd s \\
    & \leq -\limsup_{\eps\rightarrow 0}\iint_{\Pi_T} \iint_{\Pi_T}\big[(v^m)_x^2 - (u^m)_y (v^m)_x\big]\sgn'_\eps(v^m-u^m)\varphi \dd x \dd t \dd y \dd s.
  \end{split}
\end{equation}
The terms in \eqref{eq:doubling7} containing $b$ can be rewritten as follows:
\begin{align}
    & \sgn(v-u)(v^m-u^m)b(x)\varphi_x -\sgn(v-u)b'(x)u^m \varphi \nonumber\\
    & \ = \sgn(v-u)(v^m b(x) - u^m b(y))\varphi_x + \sgn(v-u) (u^m (b(y) - b(x)) \varphi)_x.\label{eq:doubling9}
\end{align}
Now, we repeat the previous steps with a simultaneous interchange of the roles of $u$ and $v$ and the roles of $(y,s)$ and $(x,t)$.
Summation of \eqref{eq:doubling7} with its respective counter part yields
\begin{equation}
  \label{eq:dounlingBIG}
  \begin{split}
    & -\iint_{\Pi_T}\iint_{\Pi_T}\big(|v-u|(\varphi_t + \varphi_s) + \sgn(v-u)\big[v^m b(x) - u^m b(y)\big](\varphi_x + \varphi_y) \\
    & \qquad + |v^m - u^m|(\varphi_{xx} + 2\varphi_{xy} + \varphi_{yy}) \\
    & \qquad + \sgn(v-u)\big[(u^m (b(y)-b(x))\varphi)_x  - (v^m (b(x)-b(y))\varphi)_y\big]\big) \dd x\dd t\dd y\dd s \\
    & \ \leq -\limsup_{\eps\rightarrow 0}\iint_{\Pi_T}\iint_{\Pi_T}|(v^m)_x - (u^m)_y|^2 \sgn'_\eps(v^m-u^m) \varphi \dd x\dd t\dd y\dd s \leq 0,
  \end{split}
\end{equation}
where we have used the identity
\begin{align*}
  \sgn(v-u)((v^m)_x - (u^m)_y) = \big[|v^m - u^m|_x + |v^m - u^m|_y\big]
\end{align*}
and integrated by parts.
We emphasize that for the estimation on the right-hand side of \eqref{eq:dounlingBIG},
the dissipation term from the definition \eqref{eq:entropyformulation} has been essential.

As usual, $\varphi=\varphi(x,t,y,s)$ is chosen in product form,
\begin{equation*}
  \varphi(x,t,y,s) = \psi\big(\frac{x+y}{2},\frac{t+s}{2}\big)\omega_\sigma\big(\frac{x-y}{2}\big)\delta_\sigma\big(\frac{t-s}{2}\big),
\end{equation*}
where $0\leq \psi \in C_c^\infty(\Pi_T)$ is a test function, $\delta_\sigma = \sigma^{-1}\delta_1(t/\sigma)$ for $\sigma>0$ and $\delta_1$ is the mollifier defined in \eqref{eq:mollifier}, and
\begin{equation*}
  \omega_\sigma(x)=\frac{1}{2\sigma} \delta_1\big(\frac{|x|^2}{\sigma^2}\big).
\end{equation*}
Accordingly, we introduce a new set of variables $(\bar x,\bar t, z,\tau)$ with
\begin{equation*}
  \bar x = \frac{x+y}{2},\quad \bar t = \frac{t+s}{2},\quad z = \frac{x-y}{2},\quad \tau = \frac{t-s}{2},
\end{equation*}
for which we have (by the usual abuse of notation)
\begin{align*}
  \partial_{\bar t}= \partial_t+\partial_s, \quad
  \partial_{\bar x}= \partial_x+\partial_y, \quad
  \partial_{\bar x\bar x}= \partial_{xx} + \partial_{yy} + 2\partial_{xy}.
\end{align*}
With the understanding that $u=u(y,s)$ and $v=v(x,t)$ while $\psi=\psi(\bar x,\bar t)$, $\omega_\sigma=\omega_\sigma(z)$ and $\delta_\sigma=\delta_\sigma(\tau)$,
the inequality \eqref{eq:dounlingBIG} can be written as
\begin{align}
  \label{eq:doubling_many}
  \begin{split}
    0 \ge J_\sigma :=
    -\iint_{\Pi_T}\iint_{\Pi_T}&\bigg[ \Big\{
    |v-u|\partial_{\bar t}\psi
    + \sgn(v-u)(v^m b(x) -u^m b(y))\partial_{\bar x}\psi
    + |v^m-u^m|\partial_{\bar x\bar x}\psi \\
    & + \sgn(v-u)\big[\big((u^m(b(y)-b(x)))_x - (v^m(b(x)-b(y)))_y\big)\psi \\
    & \qquad + u^m(b(y)-b(x))\partial_x \psi - v^m (b(x)-b(y))\partial_y \psi\big]
    \Big\}\omega_\sigma\delta_\sigma \\
    & + (b(x)-b(y))|v^m-u^m|\psi\delta_\sigma
    \partial_z \omega_\sigma\bigg] \dd x \dd t \dd y \dd s.
  \end{split}
\end{align}
By assumption, $b\in W^{1,\infty}(\setR)$.
Moreover, since $u$ and $v$ belong to $L^\infty(\Pi_T)$, and since $\delta_1'(x)\leq 0$ for $x>0$,
we can find a constant $K$ (depending on $b$)
such that
\begin{equation}
  \label{eq:b_quotient}
  \begin{split}
    (b(x)-b(y))|v^m-u^m|\partial_z\omega_\sigma
    &= (b(x)-b(y))|v^m-u^m|z\sigma^{-3}\delta_1'\big(\frac{z^2}{\sigma^2}\big)  \\
    &\leq K|v-u|\frac{z^2}{\sigma^2} \sigma^{-1}\chi_{|z|\leq 2\sigma}.
  \end{split}
\end{equation}
Now we perform the limit $\sigma\downarrow0$, which concentrates the support of $\varphi$ on the diagonals $x=y$ and $t=s$.
By Lebesgue differentiation theorem we then obtain
\begin{equation}
  \label{eq:7}
  \begin{split}
    0\ge \lim_{\sigma\downarrow0}J_\sigma
    \ge - \iint_{\Pi_T} &\bigg[ \Big\{
    |v-u|\psi_t
    + b(x)\big|v(x,t)^m-u(x,t)^m\big|\psi_x
    + \big|v(x,t)^m-u(x,t)^m|\psi_{xx} \\
    & + b'(x)\big|v(x,t)^m - u(x,t)^m\big|\psi
    \Big\}
    + K|v(x,t)-u(x,t)|\psi \bigg] \dd x \dd t .
  \end{split}
\end{equation}
Regrouping terms and using again that $b\in W^{1,\infty}(\setR)$,
we arrive at the key estimate
\begin{align}
  \label{eq:kruz2}
  -\iint_{\Pi_T}\big(|u-v|\psi_t + b|v^m-u^m|\psi_x + |v^m-u^m|\psi_{xx} ) \dd x\dd t \leq C\iint_{\Pi_T}|v-u|\psi \dd x\dd t,
\end{align}
for some $C>0$ depending on $K$ and on $b$.

At this point, we make the classical special choice for the test function $\psi$.
For any given $0<t_1<t_2<T$ and $r>0$,
let $\theta_\sigma\in C^\infty_c(]0,T[)$ be the $\sigma$-mollification of the characteristic function of the interval $[t_1,t_2]$,
and let $\phi_r\in C^\infty_c(\setR)$ be such that $\phi_r(x)=1$ for $|x|\le r$ and $\phi_r(x)=0$ for $|x|\ge r+1$.
Setting $\psi(x,t)=\phi_r(x)\theta_\sigma(t)$ yields
\begin{align*}
    &\lim_{r\to\infty}\iint_{\Pi_T} \big( b|v^m-u^m|\psi_x + |v^m-u^m|\psi_{xx}\big) \dd x \dd t  \\
    & \leq \bar C \lim_{r\to\infty} \iint_{\Pi_T\cap ||x|-r|\leq 1}\big(v+u\big)\dd x\dd t = 0
\end{align*}
by the dominated convergence theorem (with respect to $t\in[0,T]$),
because $u$ and $v$ are $t$-uniformly bounded in $L^1(\R)$.
Therefore, by sending $r\to\infty$ in \eqref{eq:kruz2}, we get
\begin{equation}\label{eq:kruz3}
    -\int_0^T \int_\R |v-u|\theta'_{\sigma}(t) \dd x\dd t \leq C\int_0^T \int_\R |v-u|\theta_{\sigma}(t) \dd x\dd t.
\end{equation}
Finally, passing to $\sigma\downarrow0$,
we obtain
\begin{equation}\label{eq:kruz4}
    \|u(t_2)-v(t_2)\|_{L^1} \leq \|u(t_1)-v(t_1)\|_{L^1} + C\int_{t_1}^{t_2} \|u(\tau)-v(\tau)\|_{L^1}\dd\tau,
\end{equation}
hence we can use the Gronwall inequality in \eqref{eq:kruz4} (integral form for measurable functions, cf. \cite{ethier}) to obtain
\begin{equation}\label{eq:kruz5}
    \|u(t_2)-v(t_2)\|_{L^1} \leq \|u(t_1)-v(t_1)\|_{L^1}(1+ C(t_2-t_1)e^{C(t_2-t_1)}).
\end{equation}
Since $u$ and $v$ are right continuous at $t=0$,
we can perform the limit $t_1\downarrow 0$ in \eqref{eq:kruz5} and obtain stability of the entropy solutions $u$, $v$ in the $L^1(\setR)$ norm.
In particular, if $v(0)=u(0)$, then $v(t)=u(t)$ for all $t>0$, which shows uniqueness.

\appendix

\section{A lemma on lower semi-continuity}\label{sec:app_LSC}
\begin{lem}
  \label{lem:Hlsc}
  Let $H\in C^1(\setR_+)$, and let $\phi\in C^0(\setR)$ be a bounded non-negative function.
  Define the functional $\Psi:\dens\to[0,\infty]$ by
  \begin{align*}
    \Psi(\eta) = \intX \phi(x)\big[\partial_xH(\eta(x))\big]^2\dd x
  \end{align*}
  whenever the integral is well-defined, and $+\infty$ otherwise.
  Then $\Psi$ is sequentially lower semi-continuous in $L^p(\setR)$, for arbitrary $p\ge1$.
\end{lem}
\begin{proof}
  Let $(\eta_n)_{n\in\setN}$ be a sequence that converges to $\eta_0$ in $L^p(\setR)$, with $\bar\Psi:=\sup_n\Psi(\eta_n)<\infty$.
  Without loss of generality, we may even assume that $\eta_n$ converges to $\eta$ almost everywhere on $\setR$.
  For every $\epsilon>0$, introduce the functional $\Psi_\epsilon:\dens\to[0,\infty]$ with
  \begin{align*}
    \Psi_\epsilon(\eta) = \int_{\Omega_\epsilon} \phi [\partial_xH(\eta)]^2\dd x
    \quad\text{where}\quad
    \Omega_\epsilon:=\big\{x\in[-\epsilon^{-1},\epsilon^{-1}]\big|\phi(x)\ge\epsilon\big\},
  \end{align*}
  with the understanding that $\Psi_\epsilon(\eta)=+\infty$ unless $H(\eta)\in H^1(\Omega_\epsilon)$.
  Since $\phi$ is non-negative,
  \begin{align*}
    \epsilon \int_{\Omega_\epsilon} [\partial_xH(\eta_n)]^2\dd x
    \le \Psi_\epsilon(\eta_n) \le \Psi(\eta_n)\le\bar\Psi.
  \end{align*}
  Thus, the functions $H(\eta_n)$ are $n$-uniformly bounded in $H^1(\Omega_\epsilon)$.
  By Alaoglu's theorem, every subsequence $(H(\eta_{n'}))_{n'\in\setN}$ contains a subsubsequence $(H(\eta_{n''}))_{n''\in\setN}$
  that converges weakly in $H^1(\Omega_\epsilon)$ to some limit $h$.
  Further, since $\Omega_\epsilon$ is compact by definition, Rellich's Lemma applies,
  and $H(\eta_{n''})$ converges to $h$ strongly in $L^2(\Omega_\epsilon)$.
  We have assumed that $\eta_n$ converges pointwise almost everywhere to $\eta_0$,
  hence $H(\eta_{n''})$ converges to $H(\eta_0)$ almost everywhere,
  and we conclude $h=H(\eta_0)\in H^1(\Omega_\epsilon)$ --- independently of the chosen subsequence.
  From here, it follows that
  \begin{align}
    \label{eq:weakl2}
    \partial_xH(\eta_n)\rightharpoonup\partial_xH(\eta_0) \quad \text{in $L^2(\Omega_\epsilon)$}.
  \end{align}
  By elementary calculations, one verifies that
  \begin{align*}
    \Psi_\epsilon(\eta_n) \ge \Psi_\epsilon(\eta_0)
    + 2 \int_{\Omega_\epsilon} \phi\big(\partial_xH(\eta_n)-\partial_xH(\eta_0)\big)\partial_xH(\eta_0)\dd x.
  \end{align*}
  Using \eqref{eq:weakl2}, the limit $n\to\infty$ provides
  \begin{align*}
    \Psi_\epsilon(\eta_0) \le \liminf_{n\to\infty} \Psi_\epsilon(\eta_n) \le \bar\Psi.
  \end{align*}
  To conclude the proof, observe that $\Psi_\epsilon(\eta_0)\to\Psi(\eta_0)$ in the limit $\epsilon\downarrow0$
  by the monotone convergence theorem.
\end{proof}

\section*{Acknowledgments}
The authors acknowledge fruitful discussions with Jos\'{e} A.~Carrillo and Boris Andreianov on this subject.
MDF is supported by the `Ramon y Cajal' Sub-programme (MICINN-RYC) of the Spanish Ministry of Science and Innovation, Ref. RYC-2010-06412
and by the by the Ministerio de Ciencia e Innovaci\'on, grant MTM2011-27739-C04-02.
Most of the present work was carried out during several visits by MDF at the Dynamical Systems Research Unit of TU Munich in 2011 and 2012,
he acknowledges support and hospitality.

\def\cprime{$'$}

\end{document}